\documentclass[11pt,a4paper]{article}

\usepackage[all]{xy}
\usepackage{latexsym}
\usepackage[margin=3cm]{geometry}
\usepackage[dvipdfmx]{graphicx}

\usepackage{here}

\usepackage{hyperref}
\usepackage{mymacro}
\usepackage{tikz}

\renewcommand{\labelenumi}{$\mathrm{(\roman{enumi})}$}
\renewcommand{\labelenumii}{$\mathrm{(\alph{enumii})}$}

\makeatletter
  
  \@addtoreset{equation}{section}
\makeatother

\title{Persistence-like distance on Tamarkin's category \\ and symplectic displacement energy}
\author{Tomohiro Asano \and Yuichi Ike}
\date{\today}

\AtEndDocument{{
		\noindent Tomohiro Asano 
		
		\noindent Graduate School of Mathematical Sciences, The University of Tokyo, 3-8-1 Komaba, Meguro-ku,
		Tokyo, 153-8914, Japan
		
		\noindent \textit{E-mail address}: \texttt{tasano@ms.u-tokyo.ac.jp}
		
		\medskip
		
		\noindent Yuichi Ike
		
		\noindent Fujitsu Laboratories Ltd., 4-1-1 Kamikodanaka,
		Nakahara-ku, Kawasaki, Kanagawa 211-8588, Japan
		
		\noindent
		\textit{E-mail address}: \texttt{yuichi.ike.1990@gmail.com},
		\texttt{ike.yuichi@fujitsu.com}
}}

\begin{document}
\maketitle

\begin{abstract}
	We introduce a persistence-like pseudo-distance on Tamarkin's category and prove that the distance between an object and its Hamiltonian deformation is at most the Hofer norm of the Hamiltonian function.
	Using the distance, we show a quantitative version of Tamarkin's non-displaceability theorem, which gives a lower bound of the displacement energy of compact subsets of cotangent bundles. 
\end{abstract}

\maketitle

\tableofcontents

\section{Introduction}

In this paper, we introduce a pseudo-distance on Tamarkin's category, inspired by the recent work by Kashiwara--Schapira~\cite{KS18persistent} on the sheaf-theoretic interpretation of the interleaving distance for persistence modules.
We also propose a new sheaf-theoretic method to estimate the displacement energy of compact subsets of cotangent bundles, which is a quantitative generalization of Tamarkin's non-displaceability theorem.
We will recall the notion of displacement energy in \cref{subsec:introdispenergy} and then state our results in \cref{subsec:mainresults}.

\subsection{Displacement energy}\label{subsec:introdispenergy}

For a given compact subset of a symplectic manifold, its displacement energy measures the minimal energy of Hamiltonian isotopies which displace the subset.
In this paper, we consider the displacement energy in the case the symplectic manifold is a cotangent bundle.
Let $M$ be a connected manifold and $I$ be an open interval containing $[0,1]$.
We denote by $T^*M$ the cotangent bundle equipped with the canonical exact symplectic structure.
A compactly supported $C^\infty$-function $H=(H_s)_{s \in I} \colon T^*M \times I \to \bR$ defines a time-dependent Hamiltonian vector field $X_H=(X_{H_s})_s$ on $T^*M$.
By the compactness of the support, $X_H$ generates a Hamiltonian isotopy $\phi^H=(\phi^H_s)_s \colon T^*M \times I \to T^*M$.
Following Hofer~\cite{Hofer90}, for a compactly supported function $H \colon T^*M \times I \to \bR$, we define
\begin{equation}
\| H \|
:=
\int_0^1 \left(\max_p H_s(p) - \min_p H_s(p) \right) ds.
\end{equation}
For compact subsets $A$ and $B$ of $T^*M$, we define their \emph{displacement energy} $e(A,B)$ by
\begin{equation}
e(A,B)
:=
\inf 
\left\{ 
\| H \| 
\;\middle|\; 
\begin{aligned}
&  \text{$H \colon T^*M \times I \to \bR$ with compact support}, \\
& A \cap \phi_1^H(B) = \emptyset
\end{aligned}
\right\}.
\end{equation}
Here $\phi_1^H$ denotes the time-one map of the Hamiltonian isotopy $\phi^H$.
Note that if $e(A,B)=+\infty$, then $A \cap \phi^H_1(B) \neq \emptyset$ for any compactly supported function $H$.
The aim of this paper is to give a lower bound of $e(A,B)$ in terms of the microlocal sheaf theory due to Kashiwara and Schapira~\cite{KS90}.

\subsection{Main results}\label{subsec:mainresults}

We shall estimate the displacement energy by introducing a pseudo-distance on Tamarkin's category $\cD(M)$.
In order to state our results, we prepare some notions.
In the sequel, let $\bfk$ be a field.
Moreover, let $X$ be a $C^\infty$-manifold.
We denote by $\Db(X)$ the bounded derived category of sheaves of $\bfk$-vector spaces.
For an object $F \in \Db(X)$, its microsupport $\MS(F)$ is defined as the set of directions in which the cohomology of $F$ cannot be extended isomorphically.
The microsupport is a closed subset of the cotangent bundle $T^*X$ and invariant under the action of $\bR_{>0}$ on $T^*X$.

In \cite{Tamarkin}, Tamarkin introduced a category $\cD(M)$ and used it to prove the non-displaceability of particular compact subsets.
The category $\cD(M)$ is defined as a quotient category of $\Db(M \times \bR)$.
For a compact subset $A$ of $T^*M$, $\cD_A(M)$ denotes the full subcategory of $\cD(M)$ consisting of objects whose microsupports are contained in the cone of $A$ in $T^*(M \times \bR)$.
For an object $F \in \cD(M)$ and $c \in \bR_{\ge 0}$ there is a canonical morphism $\tau_{0,c}(F) \colon F \to {T_c}_*F$, where $T_c \colon M \times \bR \to M \times \bR, (x,t) \mapsto (x,t+c)$.
See \cref{sec:Tamarkin} for more details.

First, using the $\bR$-direction of $M \times \bR$, we introduce the following pseudo-distance $d_{\cD(M)}$ on Tamarkin's category $\cD(M)$, which is similar to the interleaving distance for persistence modules (see \cite{CCSGGO,CdSGO16}).
Our definition is inspired by the pseudo-distances on the derived categories of sheaves on vector spaces recently introduced by Kashiwara--Schapira~\cite{KS18persistent}.
See also \cref{rem:distrelation} for their relation.
\pagebreak

\begin{definition}\label{def:introdefab}
	\
	\begin{enumerate}
		\item Let $F,G \in \cD(M)$ and $a,b \in \bR_{\ge 0}$.
		Then the pair $(F,G)$ is said to be \emph{$(a,b)$-interleaved}  
		if there exist morphisms $\alpha, \delta \colon F \to {T_a}_*G$ 
		and $\beta, \gamma : G \to {T_b}_*F$ such that
		\begin{enumerate}
			\renewcommand{\labelenumii}{$\mathrm{(\arabic{enumii})}$}
			\item $F \xrightarrow{\alpha} {T_a}_* G \xrightarrow{{T_a}_*\beta} {T_{a+b}}_*F$ is equal to $\tau_{0,a+b}(F) \colon F \to {T_{a+b}}_*F$,
			\item $G \xrightarrow{\gamma} {T_b}_* F \xrightarrow{{T_b}_*\delta} {T_{a+b}}_*G$ is equal to $\tau_{0,a+b}(G) \colon G \to {T_{a+b}}_*G$.
		\end{enumerate}
		\item For objects $F,G \in \cD(M)$, one defines 
		\begin{equation}
        	d_{\cD(M)}(F,G)
        	:=
        	\inf 
        		\left\{	a+b \in \bR_{\ge 0} 
        	\; \middle| \;
        	\begin{aligned}
        	    & a,b \in \bR_{\ge 0}, \\
            	& \text{$(F,G)$ is $(a,b)$-interleaved}
        	\end{aligned}
        	\right\},
    	\end{equation}
		and calls $d_{\cD(M)}$ the \emph{translation distance}.
	\end{enumerate}		 
\end{definition}

It might seem strange that four morphisms $\alpha,\beta,\gamma,\delta$ 
appear in (i) of the definition above.
However, to the best of the authors' knowledge, 
if we add the conditions $\alpha=\delta$ and $\beta=\gamma$, 
there is no guarantee that \cref{thm:introdist} below holds.
See also \cref{rem:fourmor}.

Now, let us consider the distance between an object in $\cD(M)$ and its Hamiltonian deformation.
Let $H \colon T^*M \times I \to \bR$ be a compactly supported Hamiltonian function.
Then, using the sheaf quantization associated with the Hamiltonian isotopy $\phi^H$ due to Guillermou--Kashiwara--Schapira~\cite{GKS} one can define a functor $\Phi^H_1 \colon \cD(M) \to \cD(M)$, which induces a functor $\Phi^H_1 \colon \cD_A(M) \to \cD_{\phi^H_1(A)}(M)$ for any compact subset $A$ of $T^*M$.
Our first result is the following.

\begin{theorem}[{see \cref{thm:GKShtpy}}]\label{thm:introdist}
	Let $G \in \cD(M)$ and $H \colon T^*M \times I \to \bR$ be a compactly supported Hamiltonian function.
	Then $d_{\cD(M)}(G,\Phi^H_1(G)) \le \|H\|$.
\end{theorem}

The outline of the proof is as follows.
First we prove that the distance between two objects is controlled  by the angle of a cone which contains the microsupport of a ``homotopy sheaf" connecting them. 
Then using the sheaf quantization associated with $\phi^H$, we can construct a homotopy sheaf $G' \in \Db(M \times \bR \times I)$ such that $G'|_{M \times \bR \times \{0\}} \simeq G, G'|_{M \times \bR \times \{1\}} \simeq \Phi^H_1(G)$ and $\MS(G') \subset T^*M \times \gamma_H$, where
\begin{equation}
\gamma_H=
\left\{ (t,s;\tau,\sigma) \;\middle|\; -\max_p H_s(p) \cdot \tau \le \sigma \le -\min_p H_s(p) \cdot \tau \right\}
\subset T^*(\bR \times I).
\end{equation}
We thus obtain the result.

Next, we use the above result to estimate the displacement energy.
One can define an internal Hom functor $\cHom^\star$ on the category $\cD(M)$, which satisfies the isomorphism
\begin{equation}
    Hom_{\cD(M)}(F,G) \simeq H^0 \RG_{M \times [0,+\infty)}(M \times \bR;\cHom^\star(F,G))
\end{equation}
for any $F,G \in \cD(M)$.
Let $q_\bR \colon M \times \bR \to \bR$ denote the projection.
Tamarkin's separation theorem asserts that if $A \cap B =\emptyset$ then $R{q_{\bR}}_*\cHom^\star(F,G) \simeq 0$ for any $F \in \cD_{A}(M)$ and $G \in \cD_{B}(M)$.
See also \cref{sec:Tamarkin}.
Using these notions, we make the following definition.

\begin{definition}
	For $F,G \in \cD(M)$, one defines
	\begin{align}
	e_{\cD(M)}(F,G)
	& :=
	d_{\cD(\pt)}(R{q_\bR}_*\cHom^\star(F,G),0) \\
	\notag & =
	\inf
	\{ c \in \bR_{\ge 0} \mid \tau_{0,c}(R{q_\bR}_* \cHom^\star(F,G))=0 \}.
	\end{align}
\end{definition}

Our main theorem is the following.

\begin{theorem}[{see \cref{thm:energy}}]\label{thm:intromainenergy}
	Let $A$ and $B$ be compact subsets of $T^*M$.
	Then, for any $F \in \cD_A(M)$ and $G \in \cD_B(M)$, one has
	\begin{equation}
	e(A,B) \ge e_{\cD(M)}(F,G).
	\end{equation}
	In particular, for any $F \in \cD_A(M)$ and $G \in \cD_B(M)$,
	\begin{equation}\label{eq:introenergyestineq}
	e(A,B)
	\!\ge\!
		\inf \{c \in \bR_{\ge 0} \mid\! \text{$\Hom_{\cD(M)}(F,G) \!\to\! \Hom_{\cD(M)}(F,{T_c}_*G)$ is zero} \}.
	\end{equation}
\end{theorem}

This theorem implies, in particular, that $\tau_{0,c}(R{q_\bR}_* \cHom^\star(F,G))$ is non-zero for any $c \in \bR_{\ge 0}$, then $A$ and $B$ are mutually non-displaceable.
In this sense, the theorem is a quantitative version of Tamarkin's non-displaceability theorem (see Tamarkin~\cite[Theorem~3.1]{Tamarkin} and Guillermou--Schapira \cite[Theorem~7.2]{GS14}).

\cref{thm:intromainenergy} is proved by Tamarkin's separation theorem and \cref{thm:introdist} as follows.
Suppose that a compactly supported Hamiltonian function $H$ satisfies $A \cap \phi_1^H(B)=\emptyset$.
Then, by Tamarkin's separation theorem, $R{q_\bR}_*\cHom^\star(F,\Phi^H_1(G)) \simeq 0$.
Thus, by fundamental properties of $d_{\cD(M)}$ and \cref{thm:introdist}, we obtain
\begin{align}
e_{\cD(M)}(F,G) 
& =
d_{\cD(\pt)}(R{q_\bR}_*\cHom^\star(F,G),0) \\
\notag & \le 
d_{\cD(M)}(\cHom^\star(F,G),\cHom^\star(F,\Phi^H_1(G))) \\
\notag & \le 
d_{\cD(M)}(G,\Phi^H_1(G)) 
\le
\| H \|.
\end{align}

As an application of \cref{thm:intromainenergy}, we prove that the displacement energy of the image of the compact exact Lagrangian immersion 
\begin{align}
S^m=\{ (x,y) \in \bR^m \times \bR \mid \|x\|^2 + y^2=1 \} & \lto T^*\bR^m \simeq \bR^{2m}, \\
\notag (x,y) & \longmapsto (x;yx)
\end{align}
is greater than or equal to $2/3$ (see \cref{eg:immersionSn}).
Using this estimate, we give a purely sheaf-theoretic proof of the following theorem of Polterovich \cite{Polterovich93}, for subsets of cotangent bundles.
Note that he proved the result for more general class of symplectic manifolds, using pseudo-holomorphic curves.

\begin{proposition}[{\cite[Corollary~1.6]{Polterovich93}}]
	Let $A$ be a compact subset of $T^*M$ whose interior is non-empty.
	Then its displacement energy is positive: $e(A,A)>0$.
\end{proposition}

\subsection{Related topics}

The interleaving distance for persistence modules is now widely used in topological data analysis (see, for example, \cite{CCSGGO,CdSGO16}).
Recently, Kashiwara--Schapira~\cite{KS18persistent} interpreted the distance as that on the derived category of sheaves.
In symplectic geometry, the notion of persistence modules was introduced by Polterovich--Shelukhin~\cite{PS16} (see also Polterovich--Shelukhin--Stojisavljevi{\'c}~\cite{PSS17}).
For barcodes of chain complexes over Novikov fields such as Floer cohomology complexes, see also Usher--Zhang~\cite{UZ16}.
Note also that \cref{thm:introdist} seems to be related to the results of Schwarz~\cite{Schwarz00} and Oh~\cite{Oh05} for continuation maps, although they did not use persistence modules.

As remarked in Tamarkin~\cite[Section~1]{Tamarkin}, for $F,G \in \cD(M)$, one can associate a submodule $H(F,G)$ of $\prod_{c \in \bR}\Hom_{\cD(M)}(F,{T_c}_*G)$, which is a module over a Novikov ring $\Lambda_{0, \mathrm{nov}}(\bfk)$ (with a formal variable $T$).
Using this module, we can express \cref{eq:introenergyestineq} in \cref{thm:intromainenergy} as
\begin{equation}
e(A,B) \ge \inf \{c \in \bR_{\ge 0} \mid \text{$H(F,G)$ is  $T^c$-torsion} \}.
\end{equation}
See \cref{rem:novikovmod} for more details.
This inequality seems to be closely related to the estimate of the displacement energy discussed in Fukaya--Oh--Ohta--Ono~\cite[Theorem~J]{FOOO09,FOOO092} and \cite[Theorem~6.1]{FOOO13}.

\subsection{Organization}

This paper is structured as follows.
In \cref{sec:preliminaries}, we recall some basics of the microlocal sheaf theory.
In \cref{sec:Tamarkin}, we review results of \cite{Tamarkin, GKS, GS14} on Tamarkin's separation theorem and sheaf quantization of Hamiltonian isotopies.
\cref{sec:displacement-energy} is the main part of the paper.
First, we introduce the translation distance $d_{\cD(M)}$ on Tamarkin's category and prove \cref{thm:introdist}.
Then we show \cref{thm:intromainenergy} and give some examples and applications.

\subsection*{Acknowledgments}

The authors would like to thank their supervisor Mikio Furuta for his encouragement and helpful advice.
They also thank Manabu Akaho for helpful discussions and Kaoru Ono for drawing their attention to relation between Tamarkin's theorem and the displacement energy.
The second named author is grateful to Vincent Humili{\`e}re, Alexandru Oancea, and Pierre Schapira for fruitful discussions.
The second named author also expresses his gratitude to IMJ-PRG and ``equipe Analyse Alg{\'e}brique" for hospitality during the preparation of this paper.
This work was partially supported by a Grant-in-Aid for JSPS Fellows 15J07993 and the Program for Leading Graduate Schools, MEXT, Japan.

\section{Preliminaries on microlocal sheaf theory}\label{sec:preliminaries}

Throughout this paper, all manifolds are assumed to be of class $C^\infty$ without boundary.
Until the end of this paper, let $\bfk$ be a field.

In this section, we recall some basics of the microlocal sheaf theory due to Kashiwara and Schapira~\cite{KS90}.
We mainly follow the notation in \cite{KS90}.

\subsection{Geometric notions (\cite[\S4.3, \S A.2]{KS90})}\label{subsec:geometric}

Let $X$ be a $C^\infty$-manifold without boundary.
For a locally closed subset $A$ of $X$, we denote by $\overline{A}$ its closure and by $\Int(A)$ its interior.
We also denote by $\Delta_X$ or simply $\Delta$ the diagonal of $X \times X$.
We denote by $\tau_X \colon TX \to X$ the tangent bundle of $X$ and by $\pi_X \colon T^*X \to X$ the cotangent bundle of $X$.
If there is no risk of confusion, we simply write $\pi$ instead of $\pi_X$.
For a submanifold $M$ of $X$, we denote by $T^*_MX$ the conormal bundle to $M$ in $X$.
In particular, $T^*_XX$ denotes the zero-section of $T^*X$.
We set $\rT X:=T^*X \setminus T^*_XX$.

Let $f \colon X \to Y$ be a morphism of manifolds.
With $f$ we associate morphisms and a commutative diagram
\begin{equation}\label{diag:fpifd}
\begin{split}
\xymatrix{
	T^*X \ar[d]_-{\pi_X} & X \times_Y T^*Y \ar[d]^-\pi \ar[l]_-{f_d} \ar[r]^-{f_\pi} & T^*Y \ar[d]^-{\pi_Y} \\
	X \ar@{=}[r] & X \ar[r]_-f & Y,
}
\end{split}
\end{equation}
where $f_\pi$ is the projection and $f_d$ is induced by the transpose of the tangent map $f' \colon TX \to X \times_Y TY$.

We denote by $(x;\xi)$ a local homogeneous coordinate system of $T^*X$.
The cotangent bundle $T^*X$ is an exact symplectic manifold with the Liouville 1-form $\alpha_{T^*X}=\langle \xi, dx \rangle$.
The antipodal map $a \colon T^*X \to T^*X$ is defined by $(x;\xi) \mapsto (x;-\xi)$.
For a subset $A$ of $T^*X$, we denote by $A^a$ its image under the map $a$.

\subsection{Microsupports of sheaves (\cite[\S5.1, \S5.4, \S6.1]{KS90})}\label{subsec:microsupport}

For a manifold $X$, we denote by $\bfk_X$ the constant sheaf with stalk $\bfk$ and by $\Db(X)=\Db(\bfk_X)$ the bounded derived category of sheaves of $\bfk$-vector spaces on $X$.
One can define Grothendieck's six operations between derived categories of sheaves $\cRHom,\allowbreak \otimes, \allowbreak Rf_*,\allowbreak f^{-1},\allowbreak Rf_!,\allowbreak f^!$ for a morphism of manifolds $f \colon X \to Y$. 
Since we work over the field $\bfk$, we simply write $\otimes$ instead of $\lten$.
Moreover for $F \in \Db(X)$ and $G \in \Db(Y)$, we define their external tensor product $F \boxtimes G \in \Db(X \times Y) $ by $F \boxtimes G :=q_X^{-1}F \otimes q_Y^{-1}G$, where $q_X \colon X \times Y \to X$ and $q_Y \colon X \times Y \to Y$ are the projections.
For a locally closed subset $Z$ of $X$, we denote by $\bfk_Z \in \Db(X)$ the constant sheaf with stalk $\bfk$ on $Z$, extended by $0$ on $X \setminus Z$.
Moreover, for a locally closed subset $Z$ of $X$ and $F \in \Db(X)$, we define
\begin{equation}
F_Z:=F \otimes \bfk_Z, \quad \RG_Z(F):=\cRHom(\bfk_Z,F).
\end{equation}
One denotes by $\omega_X \in \Db(X)$ the dualizing complex on $X$, that is, $\omega_X:=a_X^!\bfk$, where $a_X \colon X \to \pt$ is the natural morphism.
Note that $\omega_X$ is isomorphic to $\ori_X[\dim X]$, where $\ori_X$ is the orientation sheaf on $X$.
More generally, for a morphism of manifolds $f \colon X \to Y$, we denote by $\omega_f=\omega_{X/Y}:=f^!\bfk_Y \simeq \omega_X \otimes f^{-1}\omega_Y^{\otimes -1}$ the relative dualizing complex.
For $F \in \Db(X)$, we define the Verdier dual of $F$ by $\bD_XF:=\cRHom(F,\omega_X)$.

Let us recall the definition of the \emph{microsupport} $\MS(F)$ of an object $F \in \Db(X)$.

\begin{definition}[{\cite[Definition~5.1.2]{KS90}}]\label{def:microsupport}
	Let $F \in \Db(X)$ and $p \in T^*X$.
	One says that $p \not\in \MS(F)$ if there is a neighborhood $U$ of $p$ in $T^*X$ such that for any $x_0 \in X$ and any $C^\infty$-function $\varphi$ on $X$ (defined on a neighborhood of $x_0$) with $d\varphi(x_0) \in U$, one has $\RG_{\{\varphi \ge \varphi(x_0)\}}(F)_{x_0} \simeq 0$.
\end{definition}

The following properties can be checked from the definition of microsupports.
\begin{enumerate}
	\item The microsupport of an object in $\Db(X)$ is a conic (i.e., invariant under the action of $\bR_{>0}$ on $T^*X$) closed subset of $T^*X$.
	\item For an object $F \in \Db(X)$, one has $\MS(F) \cap T^*_XX=\pi(\MS(F))=\Supp(F)$.
	\item The microsupports satisfy the triangle inequality: if $F_1 \lto F_2 \lto F_3 \toone$ is a distinguished triangle in $\Db(X)$, then $\MS(F_i) \subset \MS(F_j) \cup \MS(F_k)$ for $j \neq k$.
\end{enumerate}
We also use the notation $\mathring{\MS}(F):=\MS(F) \cap \rT X=\MS(F) \setminus T^*_XX$.

\begin{example}
	(i) If $F$ is a locally constant sheaf on $X$, then $\MS(F) \subset T^*_XX$.
	Conversely, if $\MS(F) \subset T^*_XX$ then the cohomology sheaves $H^k(F)$ are locally constant for all $k \in \bZ$.
	\smallskip
	
	\noindent (ii) Let $M$ be a closed submanifold of $X$.
	Then $\MS(\bfk_M)=T^*_MX \subset T^*X$.
	\smallskip
	
	\noindent (iii) Let $\varphi$ be a $C^\infty$-function on $X$ and assume that $d\varphi(x) \neq 0$ for any $x \in \varphi^{-1}(0)$.
	Set $U:=\{x \in X \mid \varphi(x)>0\}$ and $Z:=\{x \in X\mid \varphi(x) \ge 0\}$.
	Then
	\begin{equation}
	\begin{split}
	\MS(\bfk_U)=T^*_XX|_U \cup
	\{(x;\lambda d\varphi(x)) \mid \varphi(x)=0, \lambda \le 0\}, \\
	\MS(\bfk_Z)=T^*_XX|_Z \cup
	\{(x;\lambda d\varphi(x)) \mid \varphi(x)=0, \lambda \ge 0\}.
	\end{split}
	\end{equation}
\end{example}

The following proposition is called (a particular case of) the microlocal Morse lemma.
See \cite[Proposition~5.4.17 and Corollary~5.4.19]{KS90} for more details.
The classical theory corresponds to the case $F$ is the constant sheaf~$\bfk_X$.

\begin{proposition}\label{prp:microlocalmorse}
	Let $F \in \Db(X)$ and $\varphi \colon X \to \bR$ be a $C^\infty$-function.
	Moreover, let $a,b \in \bR$ with $a<b$ or $a \in \bR, b=+\infty$.
	Assume that
	\begin{enumerate}
		\renewcommand{\labelenumi}{$\mathrm{(\arabic{enumi})}$}
		\item $\varphi$ is proper on $\Supp(F)$,
		\item  $d\varphi(x) \not\in \MS(F)$ for any $x \in \varphi^{-1}([a,b))$.
	\end{enumerate}
	Then the canonical morphism
	\begin{equation}
	\RG(\varphi^{-1}((-\infty,b));F)
	\lto
	\RG(\varphi^{-1}((-\infty,a));F)
	\end{equation}
	is an isomorphism.
\end{proposition}

Next, we shall consider bounds for the microsupports of proper direct images, non-characteristic inverse images, and $\cRHom$.

\begin{definition}
	Let $f \colon X \to Y$ be a morphism of manifolds and $A \subset T^*Y$ be a closed conic subset.
	The morphism $f$ is said to be \emph{non-characteristic} for $A$ if
	\begin{equation}
	f_\pi^{-1}(A) \cap f_d^{-1}(T^*_XX) \subset X \times_Y T^*_YY.
	\end{equation}
\end{definition}

See \eqref{diag:fpifd} for the notation $f_\pi$ and $f_d$.
In particular, any submersion from $X$ to $Y$ is non-characteristic for any closed conic subset of $T^*Y$.
Note that submersions are called smooth morphisms in \cite{KS90}.
One can show that if $f \colon X \to Y$ is non-characteristic for $A \subset T^*Y$, then $f_d f_\pi^{-1}(A)$ is a conic closed subset of $T^*X$.

\begin{theorem}[{\cite[Proposition~5.4.4 and Proposition~5.4.13]{KS90}}]\label{thm:operations}
	Let $f \colon X \to Y$ be a morphism of manifolds, $F \in \Db(X)$, and $G \in \Db(Y)$.
	\begin{enumerate}
		\item Assume that $f$ is proper on $\Supp(F)$.
		Then $\MS(Rf_*F) \subset f_\pi f_d^{-1}(\MS(F))$.
		\item Assume that $f$ is non-characteristic for $\MS(G)$.
		Then the canonical morphism $f^{-1}G \otimes \omega_{f} \to f^!G$ is an isomorphism and $\MS(f^{-1}G) \cup \MS(f^!G) \subset f_d f_\pi^{-1}(\MS(G))$.
	\end{enumerate}
\end{theorem}

For closed conic subsets $A$ and $B$ of $T^*X$, let us denote by $A+B$ the fiberwise sum of $A$ and $B$, that is,
\begin{equation}
    A+B
    :=
	\left\{
	    (x;a+b) 
	    \; \middle| \; 
	    \begin{aligned}
	        & x \in \pi(A) \cap \pi(B),\\
	        & a \in A \cap \pi^{-1}(x), b \in B \cap \pi^{-1}(x)
	    \end{aligned}
    \right\} 
    \subset T^*X.
\end{equation}

\begin{proposition}[{\cite[Proposition~5.4.14]{KS90}}]\label{prp:SStenshom}
	Let $F, G \in \Db(X)$.
	\begin{enumerate}
		\item If $\MS(F) \cap \MS(G)^a \subset T^*_XX$, then
		$\MS(F \otimes G) \subset \MS(F)+\MS(G)$.
		\item If $\MS(F) \cap \MS(G) \subset T^*_XX$, then
		$\MS(\cRHom(F,G)) \subset \MS(F)^a+\MS(G)$.
		Moreover if $F$ is cohomologically constructible (see \cite[\S3.4]{KS90} for the definition), the natural morphism $\cRHom(F,\bfk_X) \otimes G \to \cRHom(F,G)$ is an isomorphism.
	\end{enumerate}
\end{proposition}

Using microsupports, we can microlocalize the category $\Db(X)$.
Let $A \subset T^*X$ be a subset and set $\Omega=T^*X \setminus A$.
We denote by $\Db_{A}(X)$ the subcategory of $\Db(X)$ consisting of sheaves whose microsupports are contained in $A$.
By the triangle inequality, the subcategory $\Db_{A}(X)$ is a triangulated subcategory.
We set
\begin{equation}
\Db(X;\Omega):=\Db(X)/\Db_A(X),
\end{equation}
the categorical localization of $\Db(X)$ by $\Db_A(X)$.
A morphism $u \colon F \to G$ in $\Db(X)$ becomes an isomorphism in $\Db(X;\Omega)$ if $u$ is embedded in a distinguished triangle $F \overset{u}{\to} G \to H \overset{+1}{\to}$ with $\MS(H) \cap \Omega=\emptyset$.
For a closed subset $B$ of $\Omega$, $\Db_B(X;\Omega)$ denotes the full triangulated subcategory of $\Db(X;\Omega)$ consisting of $F$ with $\MS(F) \cap \Omega \subset B$.
Note that our notation is the same as in \cite{KS90} and slightly differs from that of \cite{Gu12,Gulec}.

\subsection{Kernels (\cite[\S3.6]{KS90})}\label{subsec:kernels}

For $i=1,2,3$, let $X_i$ be a manifold.
We write $X_{ij}:=X_i \times X_j$ and $X_{123}:=X_1 \times X_2 \times X_3$ for short.
We use the same symbol $q_i$ for the projections $X_{ij} \to X_i$ and $X_{123} \to X_i$.
We also denote by $q_{ij}$ the projection $X_{123} \to X_{ij}$.
Similarly, we denote by $p_{ij}$ the projection $T^*X_{123} \to T^*X_{ij}$.
One denotes by $p_{12^a}$ the composite of $p_{12}$ and the antipodal map on $T^*X_2$.

Let $A \subset T^*X_{12}$ and $B \subset T^*X_{23}$.
We set
\begin{equation}\label{eq:compset}
A \circ B
:=p_{13}(p_{12^a}^{-1}A \cap p_{23}^{-1}B) \subset T^*X_{13}.
\end{equation}
We define the operation of composition of kernels as follows:
\begin{equation}
\begin{split}
&\underset{X_2}{\circ} \colon \Db(X_{12}) \times \Db(X_{23})  \to \Db(X_{13}) \\
&(K_{12},K_{23})  \mapsto K_{12} \underset{X_2}{\circ} K_{23}
:=
R {q_{13}}_!\,(q_{12}^{-1}K_{12}\otimes q_{23}^{-1}K_{23}).
\end{split}
\end{equation}
If there is no risk of confusion, we simply write $\circ$ instead of $\underset{X_2}{\circ}$.
By \cref{thm:operations} and \cref{prp:SStenshom} we have the following.

\begin{proposition}\label{prp:SScomp}
	Let $K_{ij} \in \Db(X_{ij})$ and set $\Lambda_{ij}:=\MS(K_{ij}) \subset T^*X_{ij} \ (ij=12,23)$.
	Assume that
	\begin{enumerate}
		\renewcommand{\labelenumi}{$\mathrm{(\arabic{enumi})}$}
		\item $q_{13}$ is proper on $q_{12}^{-1}\Supp(K_{12}) \cap q_{23}^{-1}\Supp(K_{23})$,
		\item $p_{12^a}^{-1}\Lambda_{12} \cap p_{23}^{-1}\Lambda_{23} \cap (T^*_{X_1}X_1 \times T^*X_2 \times T^*_{X_3}X_3) \subset T^*_{X_{123}}X_{123}$.
	\end{enumerate}
	Then 
	\begin{equation}
	\MS(K_{12} \underset{X_2}{\circ} K_{23}) \subset
	\Lambda_{12} \circ \Lambda_{23}.
	\end{equation}
\end{proposition}

\section{Tamarkin's separation theorem and sheaf quantization of Hamiltonian isotopies}\label{sec:Tamarkin}

In what follows, until the end of the paper, let $M$ be a non-empty connected manifold without boundary.

In this section, we recall the definition of Tamarkin's category $\cD(M)$ and the separation theorem due to Tamarkin~\cite{Tamarkin}.
We can prove the non-emptiness of the intersection of two compact subsets of $T^*M$ using the theorem.
We also review the existence result of sheaf quantizations of Hamiltonian isotopies due to Guillermou--Kashiwara--Schapira~\cite{GKS}.
This enables us to consider Hamiltonian deformations in Tamarkin's category.

\subsection{Tamarkin's separation theorem (\cite{Tamarkin,GS14})}

In this subsection, we recall the definition of Tamarkin's category $\cD(M)$ and the separation theorem. 

Denote by $(x;\xi)$ a local homogeneous coordinate system on $T^*M$ and by $(t;\tau)$ the homogeneous coordinate system on $T^*\bR$.
Define the maps
\begin{gather}
\tilde{q}_1,\tilde{q}_2,s_\bR \colon M \times \bR \times \bR \lto M \times \bR, \\
\notag \tilde{q}_1(x,t_1,t_2)=(x,t_1), \ \tilde{q}_2(x,t_1,t_2)=(x,t_2), \ s_\bR(x,t_1,t_2)=(x,t_1+t_2).
\end{gather}
If there is no risk of confusion, we simply write $s$ for $s_\bR$.
We also set
\begin{equation}
i \colon M \times \bR \to M \times \bR, \ (x,t) \longmapsto (x,-t).
\end{equation}

\begin{definition}
	For $F,G \in \Db(M \times \bR)$, one sets
	\begin{align}
	F \star G & :=Rs_!(\tilde{q}_1^{-1}F \otimes \tilde{q}_2^{-1}G), \\
	\cHom^\star(F,G) & :=R\tilde{q}_{1*} \cRHom(\tilde{q}_2^{-1}F,s^!G) \\
	& \ \simeq Rs_*\cRHom(\tilde{q}_2^{-1}i^{-1}F,\tilde{q}_1^!G).
	\end{align}
\end{definition}

Note that the functor $\star$ is a left adjoint to $\cHom^\star$.

The functor
\begin{equation}
\bfk_{M \times [0,+\infty)} \star (\ast) \colon \Db(M \times \bR) \lto \Db(M \times \bR)
\end{equation}
defines a projector on the left orthogonal ${}^\perp \Db_{\{\tau \le 0\}}(M \times \bR)$.
Similarly, the functor 
\begin{equation}
\cHom^\star(\bfk_{M \times [0,+\infty)}, \ast) \colon \Db(M \times \bR) \lto \Db(M \times \bR)
\end{equation}
defines a projector on the right orthogonal $\Db_{\{\tau \le 0\}}(M \times \bR)^\perp$.
By using these projectors, Tamarkin proved that the localized category $\Db(M \times \bR;\{\tau >0\})$ is equivalent to both the left orthogonal ${}^\perp \Db_{\{\tau \le 0\}}(M \times \bR)$ and the right orthogonal $\Db_{\{\tau \le 0\}}(M \times \bR)^\perp$:
\begin{gather}
P_l:=\bfk_{M \times [0,+\infty)} \star (\ast) \colon \Db(M \times \bR;\{\tau >0\}) \simto {}^\perp \Db_{\{\tau \le 0\}}(M \times \bR), \\
\notag P_r:=\cHom^\star(\bfk_{M \times [0,+\infty)}, \ast) \colon \Db(M \times \bR;\{\tau >0\}) \simto \Db_{\{\tau \le 0\}}(M \times \bR)^\perp.
\end{gather}
Note also the inclusion ${}^\perp \Db_{\{\tau \le 0\}}(M \times \bR), \Db_{\{\tau \le 0\}}(M \times \bR)^\perp \subset \Db_{\{\tau \ge 0 \}}(M \times \bR)$.
We set $\Omega_+:=\{\tau >0\} \subset T^*(M \times \bR)$ and define the map
\begin{equation}
\begin{split}
\xymatrix@R=10pt{
	\rho \colon \Omega_+ \ar[r] & T^*M \\
	\ (x,t;\xi,\tau) \ar@{|->}[r] \ar@{}[u]|-{\hspace{7pt} \rotatebox{90}{$\in$}} & (x;\xi/\tau). \ar@{}[u]|-{\rotatebox{90}{$\in$}}
}
\end{split}
\end{equation}

\begin{definition}
	One defines 
	\begin{align}
		\cD(M)&:=\Db(M \times \bR;\Omega_+)\\
		\notag &\simeq {}^\perp \Db_{\{\tau \le 0\}}(M \times \bR) \simeq \Db_{\{\tau \le 0\}}(M \times \bR)^\perp.
	\end{align}
	For a compact subset $A$ of $T^*M$, one also defines a full subcategory $\cD_A(M)$ of $\cD(M)$ by
	\begin{equation}
	\cD_A(M):=
	\Db_{\rho^{-1}(A)}(M \times \bR;\Omega_+).
	\end{equation}
\end{definition}

For $F \in \cD(M)$, we take the canonical representative 
\begin{equation}
    P_l(F) \in {}^\perp \Db_{\{\tau \le 0\}}(M \times \bR)
\end{equation}
unless otherwise specified.
For a compact subset $A$ of $T^*M$ and $F \in \cD_A(M)$, the canonical representative $P_l(F) \in {}^\perp \Db_{\{\tau \le 0\}}(M \times \bR)$ satisfies $\MS(P_l(F)) \subset \overline{\rho^{-1}(A)}$.
Note also that if $F \in {}^\perp \Db_{\{\tau \le 0\}}(M \times \bR)$ then 
\begin{equation}
    \cHom^\star(F,G) \in \Db_{\{\tau \le 0\}}(M \times \bR)^\perp.
\end{equation}
Thus $\cHom^\star$ induces an internal Hom functor $\cHom^\star \colon \cD(M)^{\mathrm{op}} \times \cD(M) \to \cD(M)$.

\begin{remark}\label{rem:pushD}
	Let $f \colon M \to N$ be a morphism of manifolds and set $\tl{f}:=f \times \id_\bR \colon M \times \bR \to N \times \bR$.
	Then, for $F \in {}^\perp \Db_{\{\tau \le 0 \}}(M \times \bR)$ we have 
    \begin{equation}
        R\tl{f}_!F \in {}^\perp \Db_{\{\tau \le 0 \}}(N \times \bR).
    \end{equation}
	Similarly, for $G \in \Db_{\{\tau \le 0 \}}(M \times \bR)^\perp$ we have 
	\begin{equation}
	    R\tl{f}_*G \in \Db_{\{\tau \le 0 \}}(N \times \bR)^\perp.
	\end{equation}
	In other words, the morphism $f$ induces functors $\cD(M) \to \cD(N)$.
\end{remark}

\begin{proposition}[{\cite[Lemma~4.18]{GS14}}]\label{prp:morD}
	For $F,G \in \cD(M)$, there is an isomorphism
	\begin{equation}
	\Hom_{\cD(M)}(F,G) \simeq
	H^0 \RG_{M \times [0,+\infty)}(M \times \bR;\cHom^\star(F,G)).
	\end{equation}
\end{proposition}

The following separation theorem was proved by Tamarkin~\cite{Tamarkin}.

\begin{theorem}[{\cite[Theorem~3.2, Lemma~3.8]{Tamarkin} and \cite[Theorem~4.28]{GS14}}]\label{thm:separation}
	Let $A$ and $B$ be compact subsets of $T^*M$ and assume that $A \cap B=\emptyset$.
	Denote by $q_{\bR} \colon M \times \bR \to \bR$ the second projection.
	Then for any $F \in \cD_A(M)$ and $G \in \cD_B(M)$, one has $R{q_{\bR}}_* \cHom^\star(F,G) \simeq 0$.
	In particular, for any $F \in \cD_A(M)$ and $G \in \cD_B(M)$, one has $\Hom_{\cD(M)}(F,G) \simeq 0$.
\end{theorem}

\subsection{Sheaf quantization of Hamiltonian isotopies (\cite{GKS})}

We recall a result of Guillermou--Kashiwara--Schapira~\cite{GKS}, which asserts the existence of a sheaf whose microsupport coincides with the conified graph of a Hamiltonian isotopy.
The sheaf is called a sheaf quantization of the Hamiltonian isotopy.
Using sheaf quantization of Hamiltonian isotopies, we can define Hamiltonian deformations in Tamarkin's category $\cD(M)$.

Let $I$ be an open interval in $\bR$ containing $0$ and $\phi^H=(\phi^H_s)_{s \in I} \colon T^*M \times I \to T^*M$ be a Hamiltonian isotopy associated with a compactly supported Hamiltonian function $H \colon T^*M \times I \to \bR$.
Note that the Hamiltonian vector field is defined by $d\alpha_{T^*M}(X_{H_s},\ast)=-dH_{s}$ and $\phi^H$ is the identity for $s=0$.
One can conify $\phi^H$ and construct $\wh{\phi}$ such that $\wh{\phi}$ lifts $\phi^H$ as follows.
Define $\wh{H} \colon T^*M \times \rT \bR \times I \to \bR$ by  $\wh{H}_s(x,t;\xi,\tau):=\tau \cdot H_s(x;\xi/\tau)$.
Note that $\wh{H}$ is homogeneous of degree $1$, that is,  $\wh{H}_s(x,t;c\xi,c\tau)=c \cdot \wh{H}_s(x,t;\xi,\tau)$ for any $c \in \bR_{>0}$.
The Hamiltonian isotopy $\wh{\phi} \colon T^*M \times \rT \bR \times I \to T^*M \times \rT \bR$ associated with $\wh{H}$
makes the following diagram commute (recall that we have set $\Omega_+=\{\tau>0\} \subset T^*(M \times \bR)$ and $\rho \colon \Omega_+ \to T^*M, (x,t;\xi,\tau) \mapsto (x;\xi/\tau)$):
\begin{equation}\label{diag:homog}
\begin{split}
\xymatrix{
	\Omega_+ \times I \ar[r]^-{\wh{\phi}} \ar[d]_-{\rho \times \id} & \Omega_+ \ar[d]^-{\rho} \\
	T^*M \times I \ar[r]_-{\phi^H} & T^*M.
}
\end{split}
\end{equation}
Moreover there exists a $C^\infty$-function $u \colon T^*M \times I \to \bR$ such that
\begin{equation}
\wh{\phi}_s(x,t;\xi,\tau)
=
(x',t+u_s(x;\xi/\tau);\xi',\tau),
\end{equation}
where $(x';\xi'/\tau)=\phi^H_s(x;\xi/\tau)$.
By construction, $\wh{\phi}$ is a homogeneous Hamiltonian isotopy: $\wh{\phi}_s(x,t;c\xi,c\tau)=c \cdot \wh{\phi}_s(x,t;\xi,\tau)$ for any $c \in \bR_{>0}$.
See \cite[Subsection~A.3]{GKS} for more details.
We define a conic Lagrangian submanifold $\Lambda_{\wh{\phi}} \subset T^*M \times \rT \bR \times  T^*M \times \rT \bR \times T^*I$ by
\begin{equation}\label{eq:deflambdahatphi}
	\resizebox{\textwidth}{!}{$\Lambda_{\wh{\phi}}
:=
\left\{
\left(
\wh{\phi}_s(x,t;\xi,\tau), (x,t;-\xi,-\tau), (s;-\wh{H}_s \circ \wh{\phi}_s(x,t;\xi,\tau)) \right)
\; \middle| \;
\begin{aligned}
(x;\xi) & \in T^*M,  \\
(t;\tau) & \in \rT \bR, \\
s & \in I
\end{aligned}
	\right\}.$}
\end{equation}
By construction, we have
\begin{equation}
\wh{H_s} \circ \wh{\phi}_s(x,t;\xi,\tau)
=
\tau \cdot (H_s \circ \phi^H_s(x;\xi/\tau)).
\end{equation}
Note also that
\begin{align}
\notag \Lambda_{\wh{\phi}} \circ T^*_sI
&=
\left\{\left(\wh{\phi}_s(x,t;\xi,\tau), (x,t;-\xi,-\tau)\right) \; \middle| \; (x,t;\xi,\tau) \in T^*M \times \rT \bR \right\} \\
& \subset T^*M \times \rT \bR \times T^*M \times \rT\bR
\end{align}
for any $s \in I$ (see \eqref{eq:compset} for the definition of $A \circ B$).

\begin{theorem}[{\cite[Theorem~4.3]{GKS}}]\label{thm:GKS}
	In the preceding situation, there exists a unique object $K \in \Db(M \times \bR \times M \times \bR \times I)$ satisfying the following conditions:
	\begin{enumerate}
		\renewcommand{\labelenumi}{$\mathrm{(\arabic{enumi})}$}
		\item $\mathring{\MS}(K) \subset \Lambda_{\wh{\phi}}$,
		\item $K|_{M \times \bR \times M \times \bR \times \{0\}} \simeq \bfk_{\Delta_{M \times \bR}}$, where $\Delta_{M \times \bR}$ is the diagonal of $M \times \bR \times M \times \bR$.
	\end{enumerate}
	Moreover both projections $\Supp(K) \to M \times \bR \times I$ are proper. 
\end{theorem}

\begin{remark}
	In \cite[Theorem~4.3]{GKS}, it was proved that $K|_{M \times \bR \times M \times \bR \times J}$ is a bounded object for any relatively compact interval $J$ of $I$.
	Since we assume that $H$ has compact support, we find that $K \in \Db(M \times \bR \times M \times \bR \times I)$.
\end{remark}

The object $K$ is called the \emph{sheaf quantization} of $\wh{\phi}$ or associated with $\phi^H$.
Set $K_{s}:=K|_{M \times \bR \times M \times \bR \times \{s\}} \in \Db(M \times \bR \times M \times \bR)$.
Note that $\mathring{\MS}(K_s) \subset \Lambda_{\wh{\phi}} \circ T^*_sI$.
It is also proved by Guillermou--Schapira~\cite[Proposition~4.29]{GS14} that the composition with $K_s$ defines a functor
\begin{equation}
K_s \circ (\ast) \colon \cD(M) \lto \cD(M).
\end{equation}
Moreover, for $F \in \cD_A(M)$ and any $s \in I$, we have $K_s \circ F \simeq (K \circ F)|_{M \times \{s\}} \in \cD_{\phi^H_{s}(A)}(M)$.
In fact, by \cref{prp:SScomp} and \eqref{diag:homog} we get
\begin{align}
\MS(K_s \circ F) \cap \Omega_+
&\subset
(\Lambda_{\wh{\phi}} \circ T^*_sI) \circ \rho^{-1}(A)\\
\notag &=
\wh{\phi}_s(\rho^{-1}(A))
\subset
\rho^{-1}(\phi^H_s(A)).
\end{align}
In other words, the composition $K_s \circ (\ast)$ induces a functor $\cD_{A}(M) \to \cD_{\phi^H_s(A)}(M)$ for any compact subset $A$ on $T^*M$.

\section[Pseudo-distance on Tamarkin's category and displacement energy]{Pseudo-distance on Tamarkin's category and \\ displacement energy}\label{sec:displacement-energy}

In this section, we introduce a pseudo-distance $d_{\cD(M)}$ on Tamarkin's category $\cD(M)$.
We prove that the distance between an object and its Hamiltonian deformation via sheaf quantization is less than or equal to the Hofer norm of the Hamiltonian function.
Using the result, we also show a quantitative version of Tamarkin's non-displaceability theorem, which gives a lower bound of the displacement energy.

\subsection{Complements on torsion objects}

Torsion objects were introduced by Tamarkin~\cite{Tamarkin} and the category of torsion objects was systematically studied by Guillermou--Schapira~\cite{GS14}.
In this subsection, we introduce the notion of $c$-torsion for $c \in \bR_{\ge 0}$, which we will use to estimate the displacement energy.
Note that the results in this subsection are essentially due to Guillermou--Schapira~\cite{GS14}.

First we recall the microlocal cut-off lemma in a general setting.
Let $V$ be a finite-dimensional real vector space and $\gamma$ be a closed convex cone with $0 \in \gamma$ in $V$.
Define  the maps
\begin{gather}
 \tilde{q}_1,\tilde{q}_2,s_V \colon M \times V \times V \lto M \times V, \\
\notag \tilde{q}_1(x,v_1,v_2)=(x,v_1), \ \tilde{q}_2(x,v_1,v_2)=(x,v_2), \ s_V(x,v_1,v_2)=(x,v_1+v_2).
\end{gather}
For $F \in \Db(M \times V)$, the canonical morphism $\bfk_{M \times \gamma} \to \bfk_{M \times \{0\}}$ induces the morphism
\begin{equation}
R{s_V}_* (\tilde{q}_1^{-1} \bfk_{M \times \gamma} \otimes \tilde{q}_2^{-1}F)
\lto 
R{s_V}_* (\tilde{q}_1^{-1} \bfk_{M \times \{0\}} \otimes \tilde{q}_2^{-1}F)
\simeq F.
\end{equation}
The following is called the microlocal cut-off lemma due to Kashiwara--Schapira~\cite[Proposition~5.2.3]{KS90}, which is reformulated by Guillermou--Schapira~\cite[Proposition~4.9]{GS14}.
For a cone $\gamma$ with $0 \in \gamma$ in $V$, we define its polar cone $\gamma^\circ \subset V^*$ by
\begin{equation}
\gamma^\circ 
:=
\{ w \in V^* \mid \text{$\langle w, v \rangle \ge 0$ for any $v \in \gamma$} \}.
\end{equation}
We also identify $T^*V$ with $V \times V^*$.

\begin{proposition}\label{prp:microlocalcutoff}
	Let $V$ be a finite-dimensional real vector space and $\gamma$ be a closed convex cone with $0 \in \gamma$ in $V$.
	Then, for $F \in \Db(M \times V)$, $\MS(F) \subset T^*M \times V \times \gamma^\circ$ if and only if the morphism $R{s_V}_* (\tilde{q}_1^{-1} \bfk_{M \times \gamma} \otimes \tilde{q}_2^{-1}F) \to F$ is an isomorphism.
\end{proposition}

If $\Int(\gamma) \neq \emptyset$, then $\tilde{q}_1^{-1} \bfk_{M \times \gamma} \simeq \cRHom(\bfk_{M \times \Int(\gamma) \times V},\bfk_{M \times V \times V})$.
Hence, by \cref{prp:SStenshom}(ii), we have
\begin{equation}
R{s_V}_* (\tilde{q}_1^{-1} \bfk_{M \times \gamma} \otimes \tilde{q}_2^{-1}F)
\simeq
R{s_V}_* \RG_{M \times \Int(\gamma) \times V}(\tilde{q}_2^{-1}F).
\end{equation}

Now we return to the case $V=\bR$ and $\gamma=[0,+\infty)$.
Let $F \in \Db(M \times \bR)$.
Then, by \cref{prp:microlocalcutoff}, $F \in \Db_{\{\tau \ge 0\}}(M \times \bR)$ if and only if  \begin{equation}\label{eq:projectionF}
Rs_* (\tilde{q}_1^{-1} \bfk_{M \times [0,+\infty)} \otimes \tilde{q}_2^{-1}F)
\simto F.
\end{equation}
For $c \in \bR$, we define the translation map
\begin{equation}
T_c \colon M \times \bR \to M \times \bR, \quad (x,t) \longmapsto (x,t+c).
\end{equation}
For $F \in \Db_{\{\tau \ge 0\}}(M \times \bR)$, by \eqref{eq:projectionF}, we have
\begin{equation}
Rs_* (\tilde{q}_1^{-1} \bfk_{M \times [c,+\infty)} \otimes \tilde{q}_2^{-1}F)
\simto 
{T_c}_*F
\end{equation}
for any $c \in \bR$.
Hence, for $c \le d$, the canonical morphism $\bfk_{M \times [c,+\infty)} \to \bfk_{M \times [d,+\infty)}$ induces a morphism of functors from $\Db_{\{\tau \ge 0\}}(M \times \bR)$ to $\Db_{\{\tau \ge 0\}}(M \times \bR)$:
\begin{equation}
\tau_{c,d} \colon {T_c}_* \lto {T_d}_*.
\end{equation}

\begin{definition}[{cf.\ \cite{Tamarkin}}]
	Let $c \in \bR_{\ge 0}$.
	An object $F \in \Db_{\{\tau \ge 0\}}(M \times \bR)$ is said to be \emph{$c$-torsion} if the morphism $\tau_{0,c}(F) \colon F \to {T_c}_* F$ is zero.
\end{definition}

Note that a $c$-torsion object is $c'$-torsion for any $c' \ge c$.
Recall also that the category $\cD(M)=\Db(M \times \bR;\{\tau>0\})$ is regarded as a full subcategory of $\Db_{\{\tau \ge 0\}}(M \times \bR)$ via the projector $P_l \colon \Db(M \times \bR;\{\tau>0\}) \to {}^\perp \Db_{\{\tau \le 0 \}}(M \times \bR)$ or $P_r \colon \Db(M \times \bR;\{\tau>0\}) \to \Db_{\{\tau \le 0 \}}(M \times \bR)^\perp$.
Hence we can define $c$-torsion objects in $\cD(M)$.

Let $I$ be an open interval of $\bR$.
We recall a result on sheaves over $M \times \bR \times I$ due to Guillermou--Schapira~\cite{GS14}.
We denote by $(t;\tau)$ the homogeneous symplectic coordinate system on $T^*\bR$ and by $(s;\sigma)$ that on $T^*I$.
For $a,b \in \bR_{>0}$, we set
\begin{equation}
\gamma_{a,b}:=\{(\tau,\sigma) \in \bR^2 \mid -a \tau \le \sigma \le b \tau \} \subset \bR^2.
\end{equation}
Let $q \colon M \times \bR \times I \to M \times \bR$ be the projection.
We identify $T^*(\bR \times I)$ with $(\bR \times I) \times \bR^2$.

\begin{proposition}[{cf.\ \cite[Proposition~6.9]{GS14}}]\label{prp:torhtpy}
	Let $\cH \in \Db_{\{\tau \ge 0\}}(M \times \bR \times I)$ and $s_1<s_2$ be in $I$.
	Assume that there exist $a,b,r \in \bR_{>0}$ satisfying
	\begin{equation}
	\MS(\cH) \cap \pi^{-1}(M \times \bR \times (s_1-r,s_2+r)) \subset T^*M \times (\bR \times I) \times \gamma_{a,b}.
	\end{equation}
	Then $Rq_*(\cH_{M \times \bR \times [s_1,s_2)})$ is $(a(s_2-s_1)+\varepsilon)$-torsion and $Rq_*(\cH_{M \times \bR \times (s_1,s_2]})$ is $(b(s_2-s_1)+\varepsilon)$-torsion for any $\varepsilon \in \bR_{>0}$.
\end{proposition}

\begin{proof}
	The proof is essentially the same as that of \cite[Proposition~6.9]{GS14}.
	For the convenience of the reader, we give a detailed proof again.
	We only consider $Rq_*(\cH_{M \times \bR \times [s_1,s_2)})$ and omit the proof for the other case.

	\medskip
	\noindent (a)
	Choose a diffeomorphism $\psi \colon (s_1-r,s_2+r) \simto \bR$ satisfying $\psi|_{[s_1,s_2]}=\id_{[s_1,s_2]}$ and $d\psi(s) \ge 1$ for any $s \in (s_1-r,s_2+r)$.
	Set $\Psi:=\id_M \times \id_\bR \times \psi : M \times \bR \times (s_1-r,s_2+r) \simto M \times \bR \times \bR$ and $\cH':=\Psi_* \cH|_{M \times \bR \times (s_1-r,s_2+r)} \in \Db(M \times \bR \times \bR)$.
	Then, by the assumption on $\psi$, we have
	\begin{equation}\label{eq:reductionR}
	\MS(\cH')
	\subset
	T^*M \times (\bR \times \bR) \times \gamma_{a,b}
	\end{equation}
	and $Rq_*(\cH_{M \times \bR \times [s_1,s_2)}) \simeq Rq_*(\cH'_{M \times \bR \times [s_1,s_2)})$.
	Here $q$ in the right-hand side denotes the projection $M \times \bR \times \bR \to M \times \bR, (x,t,s) \mapsto (x,t)$ by abuse of notation.
	Therefore, replacing $\cH$ with $\cH'$, we may assume $I=\bR$ and \eqref{eq:reductionR}.

	\medskip
	\noindent (b)
	Set $V=\bR^2$ and denote by $s_V \colon M \times V \times V \to M \times V$ the addition map.
	By \cref{prp:microlocalcutoff}, we have
	\begin{equation}\label{eq:cutoff}
	R{s_V}_* \RG_{M \times \Int(\gamma_{a,b}^\circ) \times V}(\tilde{q}_2^{-1}\cH)
	\simeq
	\cH.
	\end{equation}
	Note that $\Int(\gamma_{a,b}^\circ)=\{(t,s) \in \bR^2 \mid -b^{-1} t < s < a^{-1} t \}$.
	Since 
	\begin{equation}
	    \MS(\bfk_{M \times \bR \times (s_1,s_2]}) \subset T^*_M M \times T^*_\bR \bR \times T^*\bR,
	\end{equation}
	\cref{prp:SStenshom}(ii) gives $\cH \otimes \bfk_{M \times \bR \times [s_1,s_2)} \simeq \RG_{M \times \bR \times (s_1,s_2]}(\cH)$.
	Combining with \eqref{eq:cutoff}, we obtain
	\begin{equation}
	Rq_*(\cH_{M \times \bR \times [s_1,s_2)})
	\simeq
	Rq_* R{s_V}_* \RG_{M \times D}(\tilde{q}_2^{-1}\cH),
	\end{equation}
	where $D=\Int(\gamma_{a,b}^\circ) \times V \cap \{(t,s,t',s') \mid s_1 < s+s' \le s_2 \}$.
	Consider the commutative diagram
	\begin{equation}
	\begin{split}
	\xymatrix{
		& M \times V \times V \ar[r]^-{s_V} \ar[d]^-{\id_M \times \tilde{q}} \ar[ld]_-{\tilde{q}_2} & M \times V \ar[d]^-{q} \\
		M \times V & M \times \bR \times V \ar[l]^-{q_2} \ar[r]_-{\tilde{s}} & M \times \bR,
	}
	\end{split}
	\end{equation}
	where $\tilde{q}(t,s,t',s')=(t,t',s'), q_2(x,t,t',s')=(x,t',s')$, and $\tilde{s}(x,t,t',s')=(x,t+t')$.
	By the adjunction of $(\id_M \times \tilde{q})_!$ and $(\id_M \times \tilde{q})^!$, we get
	\begin{align}
    	\notag Rq_*(\cH_{M \times \bR \times [s_1,s_2)})
    	& \simeq
    	R\tilde{s}_* (\id_M \times \tilde{q})_* \cRHom(\bfk_{M \times D}, (\id_M \times \tilde{q})^! q_2^{\,-1}\cH)[-1] \\
    		\label{eq:isomH}
    	& \simeq
    	R\tilde{s}_* \cRHom(\bfk_{M} \boxtimes R \tilde{q}_!\bfk_{D},  q_2^{\,-1}\cH)[-1].
	\end{align}
	Here, we used $\tilde{q}^! \simeq \tilde{q}^{-1}[1]$ for the first isomorphism.

	\medskip
	\noindent (c)
	Through the isomorphism \eqref{eq:cutoff}, $\tau_{0,c}(\cH)$ is induced by the canonical morphism \linebreak $\bfk_{\tl{T}_c( \Int(\gamma_{a,b}^\circ) \times V)} \to \bfk_{\Int(\gamma_{a,b}^\circ) \times V}$, where $\tl{T}_c(t,s,t',s')=(t+c,s,t',s')$.
	Moreover through \eqref{eq:isomH}, we find that $\tau_{0,c}(Rq_*(\cH_{M \times \bR \times [s_1,s_2)}))$ is induced by the morphism $\bfk_{\tl{T}_c(D)} \to \bfk_{D}$.
	In order to prove that $R \tilde{q}_!\bfk_{\tl{T}_c(D)} \to R \tilde{q}_!\bfk_{D}$ is zero morphism for $c>a(s_2-s_1)$, we will show that $R \tilde{q}_!\bfk_{D}$ and $R \tilde{q}_!\bfk_{\tl{T}_c(D)}$ have disjoint supports.
	
	\medskip
	\noindent (d)
	For a point $(t,t',s') \in \bR \times V$, $\tilde{q}^{-1}(t,t',s') \cap D=\emptyset$ if $t \le 0$ and
	\begin{equation}
	\tilde{q}^{-1}(t,t',s') \cap D
	=
	(s_1-s', s_2-s'] \cap (-b^{-1}t, a^{-1}t)
	\end{equation}
	if $t > 0$.
	This set is an empty set or a half closed interval if $t \not\in (a(s_1-s'),a(s_2-s')]$.
	Thus $\Supp(R \tilde{q}_! \bfk_{D})$ is contained in $\{(t,t',s') \mid t \in [a(s_1-s'), a(s_2-s')] \}$.
	Similarly, $\Supp(R \tilde{q}_! \bfk_{\tl{T}_c(D)})$ is contained in $\{(t,t',s') \mid t \in [a(s_1-s')+c,a(s_2-s')+c] \}$.
	Hence $\Supp(R \tilde{q}_! \bfk_{D})$ and $\Supp(R \tilde{q}_! \bfk_{\tl{T}_c(D)})$ are disjoint for $c>a(s_2-s_1)$.
\end{proof}

\subsection{Pseudo-distance on Tamarkin's category}\label{subsec:abisom}

In this subsection, we introduce a pseudo-distance on Tamarkin's category $\cD(M)$.
This enables us to discuss the relation between possibly non-torsion objects in $\cD(M)$.
Recall again that $\cD(M)$ is regarded as a full subcategory of $\Db_{\{\tau \ge 0 \}}(M \times \bR)$ via the projector $P_l$ or $P_r$.

\begin{definition}\label{def:defab}
	Let $F,G \in \Db_{\{\tau \ge 0 \}}(M \times \bR)$ and $a,b \in \bR_{\ge 0}$.
	\begin{enumerate}
		\item The pair $(F,G)$ is said to be \emph{$(a,b)$-interleaved} 
		if there exist morphisms $\alpha, \delta \colon F \to {T_a}_*G$ 
		and $\beta, \gamma \colon G \to {T_b}_*F$ such that
		\begin{enumerate}
			\renewcommand{\labelenumii}{$\mathrm{(\arabic{enumii})}$}
			\item $F \xrightarrow{\alpha} {T_a}_* G \xrightarrow{{T_a}_*\beta} {T_{a+b}}_*F$ is equal to $\tau_{0,a+b}(F) \colon F \to {T_{a+b}}_*F$,
			\item $G \xrightarrow{\gamma} {T_b}_* F \xrightarrow{{T_b}_*\delta} {T_{a+b}}_*G$ is equal to $\tau_{0,a+b}(G) \colon G \to {T_{a+b}}_*G$.
		\end{enumerate}
		\item $F$ is said to be \emph{$(a,b)$-isomorphic} to $G$ if there exist morphisms 
		$\alpha, \delta \colon F \to {T_a}_*G$ and $\beta, \gamma \colon G \to {T_b}_*F$ 
		satisfying (1), (2) in (i) and also 
		\begin{itemize}
			\item[(3)]  $\tau_{a,2a}(G) \circ \alpha=\tau_{a,2a}(G) \circ \delta$ and $\tau_{b,2b}(F) \circ \beta=\tau_{b,2b}(F) \circ \gamma$.
		\end{itemize}
	\end{enumerate}
\end{definition}

\begin{remark}\label{rem:fourmor}
	\begin{enumerate}
		\item It might seem strange that we do not add the conditions $\alpha=\delta$ and $\beta=\gamma$ in \cref{def:defab}. 
		However, if we add such conditions, there is no guarantee that \cref{lem:abisomtor} below holds.
		
		\item An $(a,b)$-isomorphism is indeed an isomorphism in the localized category $\cT(M):=\cD(M)/ \cN_{\mathrm{tor}}$, which is
		localized by the triangulated subcategory consisting of torsion objects
		(\cite[Definition~6.6]{GS14}).
		Let $F, G \in \cD(M)$.
		Then by a result of Guillermou--Schapira~\cite[Proposition~6.7]{GS14}, we have 
		\begin{equation}
		\Hom_{\cT(M)}(F,G) \simeq \varinjlim_{c \to +\infty} \Hom_{\cD(M)}(F,{T_c}_*G).
		\end{equation}
		Thus if $F$ is $(a,b)$-isomorphic to $G$ for some $a,b \in \bR_{\ge 0}$, then $F \simeq G$ in $\cT(M)$.
		All statements below hold if ``$(a,b)$-interleaved" is replaced by ``$(a,b)$-isomorphic", but we omit the proofs for simplicity.
	\end{enumerate}	
\end{remark}

The two notions we have introduced above are related to 
the notion of ``$a$-isomorphic" recently introduced by 
Kashiwara--Schapira~\cite{KS18persistent} and 
interleavings on persistence modules.
See \cref{rem:distrelation}.	

\begin{remark}
	Let $F, G \in \Db_{\{\tau \ge 0 \}}(M \times \bR)$ and $a,b \in \bR_{\ge 0}$.
	\begin{enumerate}
		\item The pair $(F,G)$ is $(a,b)$-interleaved if and only if 
		$(G,F)$ is $(b,a)$-interleaved. 
		\item If $(F,G)$ is $(a,b)$-interleaved, 
		then $(F,G)$ is $(a',b')$-interleaved for any $a' \ge a, b' \ge b$.
		\item $(F,0)$ is $(a,b)$-interleaved if and only if $F$ is $(a+b)$-torsion.
	\end{enumerate}
\end{remark}

\begin{lemma}\label{lem:abisomsum}
	If $(F_0,F_1)$ is $(a_0,b_0)$-interleaved
	and $(F_1,F_2)$ is $(a_1,b_1)$-interleaved,
	then $(F_0,F_2)$ is $(a_0+a_1,b_0+b_1)$-interleaved.
\end{lemma}

\begin{proof}
	By assumption, for $i=0,1$, there exist morphisms 
	\begin{equation}
	\alpha_i, \delta_i \colon F_i \to {T_{a_i}}_*F_{i+1}, 
	\quad 
	\beta_i, \gamma_i \colon F_{i+1} \to{T_{b_i}}_*F_i
	\end{equation}
	satisfying 
	\begin{equation}\label{eq:absumcond}
	{T_{a_i}}_* \beta_i \circ \alpha_i  = \tau_{0,a_i+b_i}(F_i), \quad 
	 {T_{b_i}}_* \delta_i \circ \gamma_i  = \tau_{0,a_i+b_i}(F_{i+1}).
	\end{equation}
	We set 
	\begin{equation}
	\begin{aligned}
	\alpha &:= {T_{a_0}}_* \alpha_1 \circ \alpha_0 \colon F_0 \to {T_{a_0+a_1}}_*F_2, \\
	 \beta &:= {T_{b_1}}_* \beta_0 \circ \beta_1 \colon F_2 \to {T_{b_0+b_1}}_*F_1, \\
	\gamma &:= {T_{b_1}}_* \gamma_0 \circ \gamma_1 \colon F_2 \to {T_{b_0+b_1}}_*F_1, \\
 \delta &:= {T_{a_0}}_* \delta_1 \circ \delta_0 \colon F_0 \to{T_{a_0+a_1}}_*F_2.
	\end{aligned}
	\end{equation}
	Let us consider the following commutative diagram:
	\[
	\xymatrix@C=50pt@R=18pt{
		&  & F_0 \ar[ld]_-{\alpha_0} \ar[dd]^-{\tau_{0,a_0+b_0}(F_0)}\\
		& {T_{a_0}}_*F_1 \ar[rd]^-{{T_{a_0}}_*\beta_0} \ar[dd]^-{\tau_{a_0,a_0+a_1+b_1}(F_1)} \ar[ld]_-{{T_{a_0}}_*\alpha_1}& \\
		{T_{a_0+a_1}}_*F_2\ar[rd]_-{{T_{a_0+a_1}}_*\beta_1 \quad }&&{T_{a_0+b_0}}_*F_0\ar[dd]^-{\tau_{a_0+b_0,a_0+a_1+b_0+b_1}(F_0)}\\
		& {T_{a_0+a_1+b_1}}_*F_1 \ar[rd]_-{{T_{a_0+a_1+b_1}}_*\beta_0 \qquad } & \\
		& & {T_{a_0+a_1+b_1+b_2}}_*F_0.
	}
	\]
	The two triangles in the diagram commute by \eqref{eq:absumcond}.
	Since we obtain the square by applying $\tau_{a_0, a_0+a_1+b_1}$ to $\beta_0$, it also commutes.
	Hence we have ${T_{a_0+a_1}}_* \beta \circ \alpha =\tau_{0,a_0+a_1+b_0+b_1}(F_0)$.
	Similarly, we get 
	\begin{equation}
	    {T_{b_0+b_1}}_* \delta \circ \gamma =\tau_{0,a_0+a_1+b_0+b_1}(F_2).
	\end{equation}
%	Moreover, by \eqref{eq:absumcond} again, we obtain
%	\begin{equation}
%	\begin{split}
%	& \tau_{a_0+a_1,2a_0+2a_1}(F_2) \circ \alpha \\
%	= \ &
%	\tau_{2a_0+a_1,2a_0+2a_1}(F_2) \circ \tau_{a_0+a_1,2a_0+a_1}(F_2) \circ {T_{a_0}}_* \alpha_1 \circ \alpha_0 \\
%	= \ & 
%	{T_{2a_0}}_* \tau_{a_1,2a_1}(F_2) \circ {T_{2a_0}}_* \alpha_1 \circ \tau_{a_0,2a_0}(F_1) \circ  \alpha_0	\\
%	= \ & 
%	{T_{2a_0}}_* \tau_{a_1,2a_1}(F_2) \circ {T_{2a_0}}_* \delta_1 \circ \tau_{a_0,2a_0}(F_1) \circ  \delta_0	\\
%	= \ & 
%	\tau_{a_0+a_1,2a_0+2a_1}(F_2) \circ \delta.
%	\end{split}
%	\end{equation}
%	Similarly, we get $\tau_{b_0+b_1,2b_0+2b_1}(F_0) \circ \beta=\tau_{b_0+b_1,2b_0+2b_1}(F_0) \circ \gamma$.
%	This completes the proof.	
\end{proof}

A similar argument to the proof of \cref{lem:abisomsum} shows the following lemma.

\begin{lemma}\label{lem:abisomhom}
	Let $F_0,F_1,G_0,G_1 \in \Db_{\{\tau \ge 0 \}}(M \times \bR)$ 
	and assume that $(F_0,F_1)$ is $(a_F,b_F)$-interleaved 
	and $(G_0,G_1)$ is $(a_G,b_G)$-interleaved.
	Then the pair $(\cHom^\star(F_0,G_0),\cHom^\star(F_1,G_1))$ is $(b_F+a_G,a_F+b_G)$-interleaved.
\end{lemma}

Now we define a pseudo-distance on Tamarkin's category $\cD(M)$.

\begin{definition}\label{def:distance}
	For object $F,G \in \cD(M)$, one defines 
	\begin{equation}
	d_{\cD(M)}(F,G)
	:=
	\inf 
		\left\{	a+b \in \bR_{\ge 0} 
	\; \middle| \;
	\begin{aligned}
	    & a,b \in \bR_{\ge 0}, \\
    	& \text{$(F,G)$ is $(a,b)$-interleaved}
	\end{aligned}
	\right\},
	\end{equation}
	and calls $d_{\cD(M)}$ the \emph{translation distance}.
\end{definition}

\begin{remark}\label{rem:distrelation}
	\begin{enumerate}
		\item \cref{def:defab} and \cref{def:distance} are inspired by the notion of ``$a$-isomorphic" and the convolution distance on the derived categories of sheaves on vector spaces recently introduced by Kashiwara--Schapira~\cite{KS18persistent}.
		In fact, if $M=\pt$ and $F$ and $G$ are $a$-isomorphic, 
		then $(F,G)$ is $(a,a)$-interleaved.
		Moreover, if $F$ is $(a,b)$-isomorphic to $G$, then $F$ and $G$ are $2\max\{ a,b \}$-isomorphic in the sense of Kashiwara--Schapira~\cite{KS18persistent}.
		\item The translation distance $d_{\cD(M)}$ is similar to the interleaving distance for persistence modules introduced by \cite{CCSGGO} (see also \cite{CdSGO16}).
		Their definition of ``$a$-interleaved" corresponds to \cref{def:defab} 
		with $a=b$ and $\alpha=\delta, \beta=\gamma$.
		However, as remarked by Usher--Zhang~\cite[Remark~8.5]{UZ16}, removing the restriction $a=b$ gives a better estimate of the displacement energy.
		In fact, if we restrict ourselves to $a=b$ and use the associated pseudo-distance, then we can only prove $d(G_0,G_1) \le 2 \int_{0}^{1} \|H_s \|_{\infty}\, ds$ in \cref{thm:GKShtpy} below.
	\end{enumerate}
\end{remark}

We summarize some properties of $d_{\cD(M)}$.

\begin{proposition}\label{prp:propertydD}
	Let $F,G,H, F_0,F_1,G_0,G_1 \in \cD(M)$.
	\begin{enumerate}
		\item $d_{\cD(M)}(F,G)=d_{\cD(M)}(G,F)$,
		\item $d_{\cD(M)}(F,G) \le d_{\cD(M)}(F,H) + d_{\cD(M)}(H,G)$,
		\item $d_{\cD(M)}(\cHom^\star(F_0,G_0),\cHom^\star(F_1,G_1)) \!\le\! d_{\cD(M)}(F_0,F_1) \!+\! d_{\cD(M)}(G_0,G_1)$.
	\end{enumerate}
	Moreover, let $f \colon M \to N$ be a morphism of manifolds and set $\tl{f}:=f \times \id_\bR \colon M \times \bR \to N \times \bR$.
	Regarding $F$ and $G$ as objects in the right orthogonal $\Db_{\{\tau \le 0 \}}(M \times \bR)^\perp$, one has 
	\begin{itemize}
		\item[$\mathrm{(iv)}$]  $d_{\cD(N)}(R{\tl{f}}_*F, R{\tl{f}}_*G) \le d_{\cD(M)}(F, G)$ \quad (see also \cref{rem:pushD}).
	\end{itemize}	
\end{proposition}

\begin{proof}
	(i) and (iv) follow from the definition of $d_{\cD(M)}$.
	(ii) follows from \cref{lem:abisomsum} and (iii) follows from \cref{lem:abisomhom}.
\end{proof}

\begin{example}
	Assume that $M$ is compact and 
	let $\varphi \colon M \to \bR$ be a $C^\infty$-function.
	Recall also that we assume $M$ is connected. 
	Define 
	\begin{equation}
	\begin{split}
	& Z:=\{ (x,t) \in M \times \bR \mid \varphi(x)+t \ge 0 \}, \\
	& F:=\bfk_{M \times [0,+\infty)}, \ G:=\bfk_Z \in {}^\perp \Db_{\{\tau \le 0 \}}(M \times \bR) \simeq \cD(M).
	\end{split}
	\end{equation}
	Set $a:= \max\{ \max \varphi, 0 \}, b:= -\min\{ \min \varphi, 0 \}$.
	Then there exist morphisms $\alpha \colon F \to {T_a}_*G$ and $\beta \colon G \to {T_b}_*F$ such that ${T_a}_*\beta \circ \alpha=\tau_{0,a+b}(F)$ and ${T_b}_*\alpha \circ \beta=\tau_{0,a+b}(G)$.
	This implies that $(F,G)$ is $(a,b)$-interleaved and 
	\begin{equation}
	d_{\cD(M)}(F,G) \le a+b=\max\{ \max \varphi, 0 \}-\min\{ \min \varphi, 0 \}.
	\end{equation}
	Since $\Hom_{\cD(M)}(F,{T_c}_*G) \simeq H^0 \RG_{M \times [-c,+\infty)}(M \times \bR;\cHom^\star(F,G)) \simeq 0$ for any $c<\max \varphi$ and $\Hom_{\cD(M)}(G,{T_c}_*F) \simeq 0$ for any $c< -\min \varphi$,
	the equation $d_{\cD(M)}(F,G) = a+b$ holds. 
\end{example}

\begin{example}
	Assume that $M$ is compact.
	Let $L$ be a compact connected exact Lagrangian submanifold of $T^*M$ and $f \colon L \to \bR$ be a primitive of the Liouville 1-form $\alpha_{T^*M}$, that is, a $C^\infty$-function satisfying $\alpha_{T^*M}|_L=df$.
	Define a locally closed conic Lagrangian submanifold $\wh{L}_f$ of $T^*(M \times \bR)$ by
	\begin{equation}
	\wh{L}_f
	:=
	\{
	(x,t;\tau \xi,\tau) \mid
	\tau >0, (x;\xi) \in L, t=-f(x;\xi)
	\}.
	\end{equation}
	Then by a result of Guillermou~\cite{Gu12, Gulec}, there exists an object $F_L \in \Db(M \times \bR)$ called the canonical sheaf quantization such that $\mathring{\MS}(F_L)=\wh{L}_f$ and $F_L|_{M \times \{t \}} \simeq \bfk_M$ for $t> -\min f$.
	Moreover $F_L$ can be regarded as an object in $\cD_{L}(M)$. 
	
	Now, for $i=1,2$, let $L_i$ be a compact connected exact Lagrangian submanifold of $T^*M$ and $f _i \colon L_i \to \bR$ be a primitive of the Liouville 1-form $\alpha_{T^*M}$.
	Then it is known that $L_1 \cap L_2 \neq \emptyset$ (see \cite{Ike19} for a sheaf-theoretic proof).
	For simplicity, we assume that 
	\begin{equation}
	\min_{p \in L_1 \cap L_2}(f_2-f_1) \le 0 \le \max_{p \in L_1 \cap L_2}(f_2-f_1).
	\end{equation}
	Moreover, let $F_i \in \Db(M \times \bR)$ be the canonical sheaf quantization associated with $L_i$ and $f_i$ for $i=1,2$.
	Set $a:=\max_{p \in L_1 \cap L_2}(f_2-f_1)$.
	Then, using an estimate of $\MS(\cHom^\star(F_1,F_2))$ and the microlocal Morse lemma (\cref{prp:microlocalmorse}), one can show that 
	\begin{equation}
	\Hom_{\cD(M)}(F_1,{T_a}_*F_2[k])
	\simeq
	H^k(M;\bfk_M)
	\end{equation}
	for any $k \in \bZ$.
	Thus there exists a morphism $\alpha \colon F_1 \to {T_a}_*F_2$ corresponding to $1 \in \bfk \simeq H^0(M;\bfk)$.
	Set $b:=\max_{p \in L_1 \cap L_2}(f_1-f_2)$.
	Then, similarly to the above, we obtain $\Hom_{\cD(M)}(F_2,{T_b}_*F_1) \simeq H^0(M;\bfk)$ and get a morphism $\beta \colon F_2 \to {T_b}_*F_1$ corresponding to $1 \in \bfk$.
	By construction, we find that ${T_b}_*\beta \circ \alpha=\tau_{0,a+b}(F_1)$ and ${T_a}_*\alpha \circ \beta=\tau_{0,a+b}(F_2)$.
	Thus $(F_1,F_2)$ is $(a,b)$-interleaved and 
	\begin{align}
	d_{\cD(M)}(F_1,F_2) 
	& \le 
	\max_{p \in L_1 \cap L_2}(f_2-f_1)+\max_{p \in L_1 \cap L_2}(f_1-f_2) \\
	\notag & =
	\max_{p \in L_1 \cap L_2}(f_2-f_1)-\min_{p \in L_1 \cap L_2}(f_2-f_1).
	\end{align}
\end{example}

Next, we prove that a ``homotopy sheaf" gives an $(a,b)$-interleaved pair.

\begin{lemma}\label{lem:abisomtor}
	Let $F \stackrel{u}{\lto} G \stackrel{v}{\lto} H \stackrel{w}{\lto} F[1]$ be a distinguished triangle in $\Db_{\{\tau \ge 0 \}}(M \times \bR)$ and assume that $F$ is $c$-torsion.
	Then $(G,H)$ is $(0,c)$-interleaved.
\end{lemma}

\begin{proof}
	By assumption, we have ${T_c}_*w \circ \tau_{0,c}(H)=\tau_{0,c}(F[1]) \circ w=0$.
	Hence, we get a morphism $\gamma \colon H \to {T_c}_*G$ satisfying $\tau_{0,c}(H)={T_c}_*v \circ \gamma$.
	\begin{equation}
	\begin{split}
	\xymatrix{
		F \ar[r]^-u \ar[d] & G \ar[r]^-v \ar[d] \ar@{}[rd]|(.7){\circlearrowright} & H \ar[r]^-w \ar[d] \ar[d] \ar@{-->}[ld]_-{\gamma} & F[1] \ar[d]^-0 \\
		{T_c}_* F \ar[r]_-{{T_c}_*u} & {T_c}_*G \ar[r]_-{{T_c}_*v} & {T_c}_*H \ar[r]_-{{T_c}_*w} & {T_c}_*F[1]
	}
	\end{split}
	\end{equation}
	On the other hand, since $\tau_{0,c}(G) \circ u={T_c}_*u \circ \tau_{0,c}(F)=0$, there exists a morphism $\beta \colon H \to {T_c}_*G$ satisfying $\tau_{0,c}(G)=\beta \circ v$.
	\begin{equation}
	\begin{split}
	\xymatrix{
		F \ar[r]^-u \ar[d]^-0 & G \ar[r]^-v \ar[d] \ar@{}[rd]|(.3){\circlearrowright} & H \ar[r]^-w \ar[d] \ar[d] \ar@{-->}[ld]^-{\beta} & F[1] \ar[d] \\
		{T_c}_* F \ar[r]_-{{T_c}_*u} & {T_c}_*G \ar[r]_-{{T_c}_*v} & {T_c}_*H \ar[r]_-{{T_c}_*w} & {T_c}_*F[1]
	}
	\end{split}
	\end{equation}
%	Moreover, we obtain 	
%	\begin{equation}
%	\begin{split}
%	\tau_{c,2c}(G) \circ \alpha
%	& = 
%	{T_c}_* \tau_{0,c}(G) \circ \alpha \\
%	& = 
%	{T_c}_* \delta \circ {T_c}_*v \circ \alpha \\
%	& =
%	{T_c}_* \delta \circ \tau_{0,c}(H) \\
%	& =
%	{T_c}_* \tau_{0,c}(G) \circ \delta 
%	=
%	\tau_{c,2c}(G) \circ \delta.
%	\end{split}		
%	\end{equation}
	This proves the result.
\end{proof}

\begin{proposition}\label{prp:abisomhtpy}
	Let $I$ be an open interval containing the closed interval $[0,1]$ and
	 $\cH \in \Db_{\{\tau \ge 0\}}(M \times \bR \times I)$.
	Assume that there exist continuous functions $f, g \colon I \to \bR_{\ge 0}$ satisfying
	\begin{equation}
	\MS(\cH) \subset T^*M \times \{(t,s;\tau,\sigma) \mid -f(s) \cdot \tau \le \sigma \le g(s) \cdot \tau \}.
	\end{equation}
	Then $\left(\cH|_{M \times \bR \times \{0\}},\cH|_{M \times \bR \times \{1\}} \right)$ is $\left( \int_{0}^{1} g(s) ds+\varepsilon, \int_{0}^{1} f(s) ds +\varepsilon \right)$-interleaved for any $\varepsilon \in \bR_{>0}$.
\end{proposition}

\begin{proof}
	Set $\Lambda':=\{(t,s;\tau,\sigma) \mid -f(s) \cdot \tau \le \sigma \le g(s) \cdot \tau \}$.
	Let $s_1<s_2$ be in $[0,1]$ and $\varepsilon' \in \bR_{>0}$ be an arbitrary positive number.
	Then there is $r \in \bR_{>0}$ such that
	\begin{equation}
	f(s) \le \max_{s \in [s_1,s_2]} f(s) + \frac{\varepsilon'}{2}
	\quad \text{and} \quad 
	g(s) \le \max_{s \in [s_1,s_2]} g(s)+ \frac{\varepsilon'}{2}
	\end{equation}
	for any $s \in (s_1-r,s_2+r)$, which implies
	\begin{equation}
	\Lambda' \cap \pi^{-1}(M \times \bR \times (s_1-r,s_2+r))
	\subset
	T^*M \times (\bR \times I) \times \gamma_{a + \frac{\varepsilon'}{2}, b + \frac{\varepsilon'}{2}}
	\end{equation}
	with $a=\max_{s \in [s_1,s_2]} f(s)$ and $b=\max_{s \in [s_1,s_2]} g(s)$.
	Let $q : M \!\times\! \bR \!\times\! I \!\to\! M \!\times\! \bR$ be the projection.
	By \cref{prp:torhtpy}, $Rq_*(\cH_{M \times \bR \times [s_1,s_2)})$ is $(a(s_2-s_1)+\varepsilon')$-torsion and $Rq_*(\cH_{M \times \bR \times (s_1,s_2]})$ is $(b(s_2-s_1)+\varepsilon')$-torsion.
	Hence, by Lemmas~\ref{lem:abisomsum} and \ref{lem:abisomtor}, and the distinguished triangles
	\begin{equation}
	\begin{split}
	& Rq_* (\cH_{M \times \bR \times (s_1,s_2]})
	\lto
	Rq_*(\cH_{M \times \bR \times [s_1,s_2]})
	\lto
	\cH|_{M \times \bR \times \{s_1\}}
	\toone, \\
	& Rq_*(\cH_{M \times \bR \times [s_1,s_2) })
	\lto
	Rq_*(\cH_{M \times \bR \times [s_1,s_2]})
	\lto
	\cH|_{M \times \bR \times \{s_2\}}
	\toone,
	\end{split}
	\end{equation}
	we find that $\left( \cH|_{M \times \bR \times \{s_1\}},\cH|_{M \times \bR \times \{s_2\}} \right)$ 
	is $(b(s_2-s_1)+\varepsilon',a(s_2-s_1)+\varepsilon')$-interleaved.
	Thus, by \cref{lem:abisomsum} again, 
	$\left( \cH|_{M \times \bR \times  \{0\}},\cH|_{M \times \bR \times \{1\}} \right)$ 
	is $(b_n+\varepsilon/2,a_n+\varepsilon/2)$-interleaved for any $n \in \bZ_{>0}$, where $a_n$ and $b_n$ are the Riemann sums 
	\begin{equation}
	a_n
	=
	\sum_{k=0}^{n-1} \frac{1}{n} \cdot \max_{s \in \left[ \frac{k}{n}, \frac{k+1}{n} \right]} f(s)
	\quad \text{and} \quad
	b_n
	=
	\sum_{k=0}^{n-1} \frac{1}{n} \cdot \max_{s \in \left[ \frac{k}{n}, \frac{k+1}{n} \right]} g(s).
	\end{equation}
	Since $f$ and $g$ are continuous on $I$, there is a sufficiently large $n \in \bZ_{>0}$ such that
	\begin{equation}
	a_n
	\le
	\int_{0}^{1} f(s) ds +\frac{\varepsilon}{2}
	\quad \text{and} \quad
	b_n
	\le
	\int_{0}^{1} g(s) ds +\frac{\varepsilon}{2},
	\end{equation}
	which completes the proof.
\end{proof}

Now, let us consider the distance between Hamiltonian isotopic objects in $\cD(M)$.
Using sheaf quantization of Hamiltonian isotopies (\cref{thm:GKS}), we can define Hamiltonian deformations in $\cD(M)$.
From now on, until the end of this section, 
we assume moreover that the dimension of $M$ is greater than $0$
and fix an open interval $I$ containing $[0,1]$.
For a compactly supported Hamiltonian function $H=(H_s)_s \colon T^*M \times I \to \bR$, following Hofer~\cite{Hofer90}, we define 
\begin{equation}
\begin{aligned}
E_+(H)
& :=
\int_0^1 \max_{p} H_s(p) ds, 
\qquad 
E_-(H)
:=
-\int_0^1 \min_{p} H_s(p) ds, \\
\| H \|
& :=
E_+(H)+E_-(H)
=
\int_0^1 \left(\max_p H_s(p) - \min_p H_s(p) \right) ds.
\end{aligned}
\end{equation}

\begin{theorem}\label{thm:GKShtpy}
	Let $H=(H_s)_s \colon T^*M \times I \to \bR$ be a compactly supported Hamiltonian function and denote by $\phi^H$ the Hamiltonian isotopy generated by $H$.	
	Let $K \in \Db(M \times \bR \times M \times \bR \times I)$ be the sheaf quantization associated with $\phi^H$.
	Moreover, let $G \in \cD(M)$, and set $G':=K \circ G \in \Db(M \times \bR \times I)$ and $G_s:=G'|_{M \times \bR \times \{s\}} \in \cD(M)$ for $s \in I$.
	Then $(G_0,G_1)$ is $( E_-(H)+\varepsilon, E_+(H)+\varepsilon )$-interleaved 
	for any $\varepsilon \in \bR_{>0}$.
	In particular, $d_{\cD(M)}(G_0,G_1) \le \| H \|$.
\end{theorem}

\begin{proof}
	By \cref{prp:SScomp} and \eqref{eq:deflambdahatphi}, we get
	\begin{equation}
	\MS(G') \subset T^*M \times \left\{ (t,s;\tau,\sigma) \;\middle|\;  -\max_p H_s(p) \cdot \tau \le \sigma \le -\min_p H_s(p) \cdot \tau \right\}.
	\end{equation}
	Thus the result follows from \cref{prp:abisomhtpy}.
\end{proof}

\subsection{Displacement energy}

In this subsection, we prove a quantitative version of Tamarkin's non-displaceability theorem, which gives a lower bound of the displacement energy.

For compact subsets $A$ and $B$ of $T^*M$, their \emph{displacement energy} $e(A,B)$ is defined by
\begin{equation}
e(A,B)
:=
\inf 
\left\{ 
\| H \| 
\;\middle|\; 
\begin{aligned}
&  \text{$H \colon T^*M \times I \to \bR$ with compact support}, \\
& A \cap \phi_1^H(B) = \emptyset
\end{aligned}
\right\}.
\end{equation}
For a compact subset $A$ of $T^*M$, set $e(A)=e(A,A)$.

We give a sheaf-theoretic lower bound of $e(A,B)$.
For that purpose, we make the following definition.

\begin{definition}\label{def:sheafenergy}
	For $F,G \in \cD(M)$, one defines
	\begin{align}
	e_{\cD(M)}(F,G)
	& :=
	d_{\cD(\pt)}(R{q_\bR}_*\cHom^\star(F,G),0) \\
	\notag & = 
	\inf \{	c \in \bR_{\ge 0} \mid \text{$R{q_\bR}_* \cHom^\star(F,G)$ is $c$-torsion} \}.	
	\end{align}
\end{definition}

Note that by \cref{prp:morD}, for $F,G \in \cD(M)$ we have
\begin{equation}
	e_{\cD(M)}(F,G)
	\ge 
	\inf \{c \in \bR_{\ge 0} \mid \text{$\Hom_{\cD(M)}(F,G) \to \Hom_{\cD(M)}(F,{T_c}_*G)$ is zero} \}.
\end{equation}

\begin{theorem}\label{thm:energy}
	Let $A$ and $B$ be compact subsets of $T^*M$.
	Then, for any $F \in \cD_A(M)$ and $G \in \cD_B(M)$, one has
	\begin{equation}
	e(A,B) \ge e_{\cD(M)}(F,G).
	\end{equation}
	In particular, for any $F \in \cD_A(M)$ and $G \in \cD_B(M)$,
	\begin{equation}\label{eq:enersyestimateHom}
	e(A,B)
	\ge
		\inf \{c \in \bR_{\ge 0} \mid \text{$\Hom_{\cD(M)}(F,G) \to \Hom_{\cD(M)}(F,{T_c}_*G)$ is zero} \}.
	\end{equation}
\end{theorem}

\begin{proof}
	Suppose that a compactly supported Hamiltonian function $H \colon T^*M \times I \to \bR$ satisfies $A \cap \phi_1^H(B)=\emptyset$.
	Let $K \in \Db(M \times \bR \times M \times \bR \times I)$ be the sheaf quantization associated with $\phi^H$ and define $G':=K \circ G \in \Db(M \times \bR \times I)$ and $G_s:=G'|_{M \times \bR \times \{s\}} \in \cD(M)$ for $s \in I$ as in \cref{thm:GKShtpy}.
	Since $G_1 \in \cD_{\phi^H_1(B)}(M)$, Tamarkin's separation theorem (\cref{thm:separation}) implies $R{q_\bR}_* \cHom^\star(F,G_1) \simeq 0$.
	On the other hand, by \cref{thm:GKShtpy}, we have $d_{\cD(M)}(G_0,G_1) \le \| H \|$.
	Hence, by \cref{prp:propertydD}, we obtain
	\begin{align}
	e_{\cD(M)}(F,G) 
	& = 
	d_{\cD(\pt)}(R{q_\bR}_*\cHom^\star(F,G_0),0) \\
	\notag & \le 
	d_{\cD(M)}(\cHom^\star(F,G_0),\cHom^\star(F,G_1)) \\
	\notag & \le 
	d_{\cD(M)}(G_0,G_1) 
	\le
	\| H \|,
	\end{align}
	which proves the theorem.
\end{proof}

We list some properties of $e_{\cD(M)}$.

\begin{proposition}
	Let $F,G \in \cD(M)$.
	\begin{enumerate}
		\item $e_{\cD(M)}(F,G) \le e_{\cD(M)}(F,F)$ 
		and $e_{\cD(M)}(F,G) \le e_{\cD(M)}(G,G)$.
		\item Assume that $F$ and $G$ are cohomologically constructible as objects in ${}^\perp \Db_{\{\tau \le 0 \}}(M \times \bR) \subset \Db(M \times \bR)$.
		Then 
		\begin{equation}
		    e_{\cD(M)}(F,G)=e_{\cD(M)}(i_*\bD_{M \times \bR} G,i_*\bD_{M \times \bR} F).
		\end{equation}
		\item Assume that there exist compact subsets $A$ and $B$ of $T^*M$ such that $F \in \cD_{A}(M)$ and $G \in \cD_{B}(M)$.
		Let $\phi^H \colon T^*M \times I \to T^*M$ be a Hamiltonian isotopy with compact support and $K \in \Db(M \times \bR \times M \times \bR \times I)$ be the sheaf quantization associated with $\phi^H$.
		Set $F':=K \circ F, G':=K \circ G$ and $F_s:=F'|_{M \times \bR \times \{s\}}, G_s:=G'|_{M \times \bR \times \{s\}}$ for $s \in I$.
		Then \begin{equation}
		    e_{\cD(M)}(F,G)=e_{\cD(M)}(F_s,G_s)
		\end{equation}
		for any $s \in I$.
	\end{enumerate}
\end{proposition}

\begin{proof}
	(i)
	First note that for any $c \in \bR_{\ge 0}$, 
	we have the following commutative diagram:
	\begin{equation}\label{eq:diagHomstar}
	\begin{aligned}
	\xymatrix{
		& \cHom^\star(F,G) \ar[d] \ar[ld] \ar[rd] \\
		\cHom({T_{-c}}_*F,G) \ar[r]_-{\sim} & 
		{T_c}_* \cHom^\star(F,G) &
		\cHom^\star(F,{T_c}_*G). \ar[l]^-{\sim}
	}
	\end{aligned}
	\end{equation}	
	Assume that the morphism 
	\begin{align}
		\notag \tau_{0,c}(R{q_\bR}_* \cHom^\star(F,F)) \colon R{q_\bR}_* \cHom^\star(F,F) \lto & \ {T_c}_*R{q_\bR}_* \cHom^\star(F,F) \\
	\simeq & \ R{q_\bR}_* \cHom^\star({T_{-c}}_*F,F)
	\end{align}
	is zero.
	Then the induced morphism 
	$\Hom_{\cD(M)}(F,F) \to \Hom_{\cD(M)}({T_{-c}}_*F,F)$ 
	is also zero by \cref{prp:morD}.
	Thus  $\tau_{-c,0}(F)=0$ as the image of $\id_F$ under the morphism.
	By the commutativity of \eqref{eq:diagHomstar}, 
	$\tau_{0,c}(R{q_\bR}_* \cHom^\star(F,G))$ is zero.
	This proves the first inequality.
	The proof for the second one is similar.

	\medskip
	\noindent (ii)
	First, we show that $i_* \bD_{M \times \bR} \colon \Db(M \times \bR) \to \Db(M \times \bR)$ induces a functor $\cD(M) \simeq {}^\perp \Db_{\{\tau \le 0 \}}(M \times \bR) \to \Db_{\{\tau \le 0 \}}(M \times \bR)^\perp \simeq \cD(M)$.	
	Let $F \in {}^\perp \Db_{\{\tau \le 0 \}}(M \times \bR)$ and $S \in \Db_{\{\tau \le 0 \}}(M \times \bR)$.
	Then we have 
	\begin{align}
	\Hom_{\Db(M \times \bR)}(S,i_*\bD_{M \times \bR}F)
	& \simeq 
	\Hom_{\Db(M \times \bR)}(i_*S,\cRHom(F,\omega_{M \times \bR})) \\
	\notag & \simeq 
	\Hom_{\Db(M \times \bR)}(i_*S \otimes F, \omega_{M \times \bR}) \\
	\notag & \simeq 
	\Hom_{\Db(M \times \bR)}(F,\cRHom(i_*S,\omega_{M \times \bR})). 
	\end{align}
	By \cref{thm:operations} and \cref{prp:SStenshom}, $\cRHom(i_*S,\omega_{M \times \bR}) \in \Db_{\{\tau \le 0 \}}(M \times \bR)$.
	Hence $\Hom_{\Db(M \times \bR)}(S,i_*\bD_{M \times \bR}F) \simeq 0$, which implies 
	\begin{equation} 
	    i_*\bD_{M \times \bR}F \in \Db_{\{\tau \le 0 \}}(M \times \bR)^\perp.
	\end{equation}
	
	Now, assume that $F,G \!\in\! {}^\perp \Db_{\{\tau \le 0 \}}(M \!\times\! \bR) $ are cohomologically con\-struct\-ible.
	Then we have 
	\begin{align}
	\cHom^\star(F,G)
	& \simeq
	Rs_* \cRHom(\tilde{q}_2^{-1} i^{-1}F, \tilde{q}_1^!G) \\
	\notag & \simeq 
	Rs_* \cRHom(\bD_{M \times \bR} \tilde{q}_1^!G, \bD_{M \times \bR} \tilde{q}_2^{-1} i^{-1}F) \\
	\notag & \simeq 
	Rs_* \cRHom(\tilde{q}_1^{-1} \bD_{M \times \bR} G, \tilde{q}_2^! i^{-1} \bD_{M \times \bR} F)\\
\notag 	& \simeq 
	\cHom^\star(i_*\bD_{M \times \bR} G, i_* \bD_{M \times \bR} F),
	\end{align}		
	which proves the equality.

	\medskip
	\noindent (iii)
	It is enough to show that $R{q_\bR}_* \cHom^\star(F,G) \simeq R{q_\bR}_* \cHom^\star(F_s,G_s)$ for any $s \in I$.
	For a compact subset $C$ of $T^*M$, define $\mathrm{Cone}_{H}(C) \subset T^*(M \times I) \times \bR$ by
	\begin{equation}
	\mathrm{Cone}_{H}(C) :=
	\overline{
		\left\{
		(x',s;\xi',-\tau \cdot H_s(x';\xi'/\tau), \tau) 
		\; \middle| \;
		\begin{aligned}
		    & \tau>0, (x;\xi/\tau) \in C, \\
		    & (x';\xi'/\tau)=\phi^H_s(x;\xi/\tau)
		\end{aligned}
		\right\}}.
	\end{equation}
	Denote by $\hat{\pi} \colon T^*(M \times I \times \bR) \simeq T^*(M \times I) \times T^*\bR \to T^*(M \times I) \times \bR$ the projection.
	Then, by \cref{prp:SScomp} and \eqref{eq:deflambdahatphi}, we have 
	\begin{equation}
	\MS(F') \subset \hat{\pi}^{-1}(\mathrm{Cone}_{H}(A)), \quad
	\MS(G') \subset \hat{\pi}^{-1}(\mathrm{Cone}_{H}(B)).
	\end{equation}
	Moreover, let $q_{I \times \bR} \colon M \times I \times \bR \to I \times \bR$ be the projection.
	Note that $q_{I \times \bR}$ is proper on $\Supp(\cHom^\star(F',G'))$, where $\cHom^\star$ denotes the internal Hom functor on $\cD(M \times I)$.
	Then, by \cite[Proposition~4.13 and Lemma~4.7]{GS14} and \cref{thm:operations}, we obtain
	\begin{equation}
	\MS(R{q_{I \times \bR}}_*\cHom^\star(F',G'))
	\subset 
	\{ (s,t;0,\tau) \mid \tau \ge 0 \} \subset T^*(I \times \bR).
	\end{equation}
	Since $I$ is contractible, there exists $S \in \Db(\bR)$ such that 
	\begin{equation}
	    R{q_{I \times \bR}}_*\cHom^\star(F',G') \simeq q'^{-1}S,
	\end{equation}
	where $q' \colon I \times \bR \to \bR$ is the projection.
	Finally, by \cite[Corollary~4.15]{GS14}, for any $s \in I$, we have 
	\begin{equation}
	R{q_{I \times \bR}}_*\cHom^\star(F',G')|_{\{s\} \times \bR} 
	\simeq 
	R{q_{\bR}}_*\cHom^\star(F_s,G_s),
	\end{equation}	
	which completes the proof.	
\end{proof}

\begin{remark}
	Assume that $F,G \in \cD(M) \simeq {}^\perp \Db_{\{\tau \le 0 \}}(M \times \bR)$ are constructible and have compact support.
	Then $R{q_\bR}_*\cHom^\star(F,G)$ is also constructible object with compact support and $\MS(R{q_\bR}_*\cHom^\star(F,G)) \subset \{ \tau \ge 0 \}$.
	By the decomposition result for constructible sheaves on $\bR$ due to Guillermou~\cite[Corollary~7.3]{Gu16} (see also \cite[Subsection~1.4]{KS18persistent}), there exist a finite family of half-closed intervals $\{[b_i,d_i) \}_{i \in I}$ and $n_i \in \bZ \ (i \in I)$ such that 
	\begin{equation}
	R{q_\bR}_*\cHom^\star(F,G)
	\simeq 
	\bigoplus_{i \in I} \bfk_{[b_i,d_i)}[n_i].
	\end{equation}
	Using this decomposition, we find that $e_{\cD(M)}(F,G)=\max_{i \in I} (d_i - b_i)$ is the length of the longest barcodes of $R{q_\bR}_*\cHom^\star(F,G)$ in the sense of Kashiwara--Schapira~\cite{KS18persistent}.
\end{remark}

\begin{remark}\label{rem:novikovmod}
	Let $F,G \in \cD(M)$.
	As remarked by Tamarkin~\cite[Section~1]{Tamarkin}, we can associate a module $H(F,G)$ over a Novikov ring $\Lambda_{0,\mathrm{nov}}(\bfk)$ as follows.
	We define
	\begin{equation}
	\Lambda_{0,\mathrm{nov}}(\bfk)
	:=
	\left\{ 
	\sum_{i=1}^\infty c_i T^{\lambda_i} 
	\;\middle| \; 
	c_i \in \bfk, \lambda_i \in \bR_{\ge 0}, \lambda_1 < \lambda_2<\cdots, \lim_{i \to \infty} \lambda_i =+\infty 
	\right\}.
	\end{equation} 
	We also define a submodule $H(F,G)$ of $\prod_{c \in \bR} \Hom_{\cD(M)}(F,{T_c}_*G) $ by
	\begin{equation}
	\left\{ 
	(h_c)_c \in \prod_{c \in \bR} \Hom_{\cD(M)}(F,{T_c}_*G) 
	\;\middle|\;
		\resizebox{.5\textwidth}{!}{$\begin{aligned}
	& \exists \, (c_i)_{i=1}^\infty \subset \bR, c_1<c_2< \cdots, \lim_{i \to \infty} c_i =+\infty \\
	& \text{such that $h_c=0$ for any $c \not\in \bigcup_{i=1}^\infty \{c_i\}$}		
		\end{aligned}$}
	\right\}.
	\end{equation}
	For $c \in \bR$ and $\lambda \in \bR_{\ge 0}$, there is the canonical morphism
	\begin{equation}
	    \tau_{c,c+\lambda} \colon \Hom_{\cD(M)}(F,{T_c}_*G) \to \Hom_{\cD(M)}(F,{T_{c+\lambda}}_*G)
	\end{equation}
	induced by $\tau_{c,c+\lambda}(G) \colon {T_c}_*G \to {T_{c+\lambda}}_*G$.
	Using this morphism, we can equip $H(F,G)$ with an action of $T^\lambda$ by $T^\lambda \cdot (h_c)_c:=(\tau_{c,c+\lambda}(h_c))_c$.
	We thus find that the Novikov ring $\Lambda_{0,\mathrm{nov}}(\bfk)$ acts on $H(F,G)$.
	
	\noindent(i) Using the $\Lambda_{0,\mathrm{nov}}(\bfk)$-module $H(F,G)$, we can express \eqref{eq:enersyestimateHom} in \cref{thm:energy} as
	\begin{equation}
	e(A,B) \ge \inf \{c \in \bR_{\ge 0} \mid \text{$H(F,G)$ is  $T^c$-torsion} \}
	\end{equation}
	for any $F \in \cD_{A}(M)$ and $G \in \cD_{B}(M)$.
	This inequality seems to be related to the estimate of the displacement energy by Fukaya--Oh--Ohta--Ono~\cite[Theorem~J]{FOOO09,FOOO092} and \cite[Theorem~6.1]{FOOO13}.	
	
	\noindent(ii) We denote by $\Lambda_{\mathrm{nov}}(\bfk)$ the fraction field of $\Lambda_{0,\mathrm{nov}}(\bfk)$.
	Then, for any $F, G \in \cD(M)$, we have 
	\begin{equation}
	H(F,G) \otimes_{\Lambda_{0,\mathrm{nov}}(\bfk)} \Lambda_{\mathrm{nov}}(\bfk)
	\simeq 
	\Hom_{\cT(M)}(F,G) \otimes_{\bfk} \Lambda_{\mathrm{nov}}(\bfk)
	\end{equation}
	See \cref{rem:fourmor}(ii) for the category $\cT(M)$.
	Note also that $\cT(M)$ is invariant under Hamiltonian deformations by \cref{thm:GKShtpy} and \cref{rem:fourmor}(ii).
\end{remark}

\subsection{Examples and applications}

In this subsection, we give some examples to which \cref{thm:energy} is applicable. 

The first two examples, \cref{eg:immersionSn} and \cref{eg:immersionSn2}, treat exact Lagrangian immersions. 

\begin{example}\label{eg:immersionSn}
	Consider $T^*\bR^m \simeq \bR^{2m}$ and denote by $(x;\xi)$ the homogeneous symplectic coordinate system.
	Let $L\!=\!S^m\!=\!\{ (x,y) \!\in\! \bR^m \!\times\! \bR \mid\! \|x\|^2 \!+\! y^2=1 \}$ and consider the exact Lagrangian immersion
	\begin{equation}
	\iota \colon L \lto T^*\bR^m,
	\quad
	(x,y) \longmapsto (x;yx).
	\end{equation}
	Setting $f \colon L \to \bR, f(x,y):=-\frac{1}{3} y^3$, we have $df=\iota^* \alpha_{T^*\bR^m}$.
	We define a locally closed subset $Z$ of $\bR^m \times \bR$ by
	\begin{equation}
	Z:=
	\left\{
	(x,t) \in \bR^m \times \bR \; \middle| \;
	\| x \| \le 1, -\frac{1}{3}(1-\| x \|^2)^{\frac{3}{2}} \le t <\frac{1}{3}(1-\| x \|^2)^{\frac{3}{2}}
	\right\}
	\end{equation}
	and $F:= \bfk_Z \in \Db(\bR^m \times \bR)$.
	\begin{figure}[H]
		\begin{minipage}{0.49\hsize}
			\begin{center}
				\begin{tikzpicture}[scale=2]
				\draw[->] (-1.5,0) -- (1.5,0) node[below] {$x$};
				\draw[->] (0,-0.7) -- (0,0.7) node[left]  {$\xi$};
				\draw[thick,-] plot[domain=0:{2*pi}, variable=\t, smooth] ({sin(\t r)},{sin (\t r) * cos(\t r)});
				\path (1,0.33) node[above right] {$\iota(L)$};
				\end{tikzpicture}
				\caption{$\iota(L)$ in the case $m=1$.}
			\end{center}
		\end{minipage}
		\begin{minipage}{0.49\hsize}
			\begin{center}
				\begin{tikzpicture}[scale=2]
				\filldraw [fill=gray!50] [thick,dashed] plot[domain=0:2*pi, variable=\t, smooth] ({sin(\t r)},{0.33* cos(\t r)^3});
				\draw[thick,-] plot[domain=-0.5*pi:0.5*pi, variable=\t, smooth] ({sin(\t r)},{-0.33* cos(\t r)^3});
				\path (0.5,0.25) node[above right] {$Z$};
				\draw[->] (-1.5,0) -- (1.5,0) node[below] {$x$};
				\draw[->] (0,-0.7) -- (0,0.7) node[left]  {$t$};
				\path (0,-0.33) node[below right] {$-\frac{1}{3}$};
				\path (0,0.33) node[above right] {$\frac{1}{3}$};
				\end{tikzpicture}
				\caption{$Z$ in the case $m=1$.}
			\end{center}
		\end{minipage}
	\end{figure}
	\noindent The object $F$ is in ${}^\perp \Db_{\{\tau \le 0\}}(\bR^m \times \bR)$ and can be regarded as an object in $\cD_{\iota(L)}(\bR^m)$.
	For this object $F$, we find that
	\begin{equation}
	\Hom_{\cD(\bR^m)}(F,{T_c}_*F)
	\simeq
	\Hom_{\Db(\bR^m \times \bR)}(F,{T_c}_*F)
	\simeq
	\begin{cases}
	\bfk & \left(0 \le c < \frac{2}{3} \right) \\
	0 & \left(c \ge \frac{2}{3} \right)
	\end{cases}
	\end{equation}
	and the induced morphism $\Hom_{\cD(\bR^m)}(F,F) \to \Hom_{\cD(\bR^m)}(F,{T_c}_*F)$ is the identity for any $0 \le c < 2/3$.
	Hence, we obtain $e(\iota(L)) \ge e_{\cD(\bR^m)}(F,F) \ge 2/3$ by \cref{thm:energy}.
	This is the same estimate as that of Akaho~\cite{Akaho}.
	If $m=1$, it is known that $e(\iota(L))=4/3$ by the use of Hofer-Zehnder capacity.
\end{example}

Using the example above, we can recover the following result of Polterovich~\cite{Polterovich93}, for subsets of cotangent bundles.

\begin{proposition}[{\cite[Corollary~1.6, see also the first remark in p.~360]{Polterovich93}}]
	Let $A$ be a compact subset of $T^*M$ whose interior is non-empty.
	Then its displacement energy is positive: $e(A)>0$.
\end{proposition}

\begin{proof}
	Take a symplectic diffeomorphism $\psi \colon T^*M \to T^*M$ such that $T^*_MM \cap \Int(\psi(A)) \neq \emptyset$.
	Since $e(\psi(A))=e(A)$, we may assume $T^*_MM \cap \Int(A) \neq \emptyset$ from the beginning.
	Take a point $x_0 \in T^*_MM \cap \Int(A)$ and a local coordinate system $x=(x_1,\dots,x_m)$ on $M$ around $x_0$.
	Denote by $(x;\xi)$ the associated local homogeneous symplectic coordinate system on $T^*M$.
	Using the coordinates, for $\varepsilon \in \bR_{>0}$ we define $\iota_\varepsilon \colon S^m \to T^*M$ by $(x,y) \mapsto (\varepsilon x, \varepsilon y x)$ as in \cref{eg:immersionSn}.
	Then, there is a sufficiently small $\varepsilon \in \bR_{>0}$ such that the image $\iota_\varepsilon(S^m)$ is contained in $\Int(A)$.
	As in \cref{eg:immersionSn}, we define $F:=\bfk_{Z_\varepsilon} \in \cD_{\iota_\varepsilon(S^m)}(\bR^m)$, where
	\begin{equation}
	Z_\varepsilon:=
	\left\{
	(z,t) \in \bR^m \times \bR \; \middle| \;
	\| z \| \le \varepsilon, -\frac{1}{3\varepsilon}(\varepsilon^2-\| z \|^2)^{\frac{3}{2}} \le t <\frac{1}{3\varepsilon}(\varepsilon^2-\| z \|^2)^{\frac{3}{2}}
	\right\}.
	\end{equation}
	Moreover we define $G \in \cD_{\iota_\varepsilon(S^m)}(M)$ as the zero extension of $F$ to $M \times \bR$.
	By monotonicity of the displacement energy and a similar argument to \cref{eg:immersionSn}, we have
	\begin{equation}
	e(A) \ge e(\iota_\varepsilon(S^m)) \ge e_{\cD(M)}(G,G) \ge \frac{2}{3}\varepsilon^2 >0.
	\end{equation}
\end{proof}

For the next explicit example, our estimate is better than Akaho's estimate~\cite{Akaho}. 

\begin{example}\label{eg:immersionSn2}
	Let $\varphi \colon [0,1] \to (0,1]$ be a $C^\infty$-function satisfying the following two conditions:
	(1) $\varphi \equiv 1$ near $0$,
	(2) $\varphi(r)=r$ on $[1/2,1]$.
	Set $S^m=\{ (x,y) \in \bR^m \times \bR \mid \|x\|^2+y^2=1 \}$ and consider the exact Lagrangian immersion
	\begin{equation}
	\iota \colon
	S^m \lto T^*\bR^m,
	\quad
	(x,y) \longmapsto
	\left(x, \left(\varphi(\|x\|)y - \frac{\varphi'(\|x\|)}{3\|x\|} y^3 \right) \cdot x \right). 
	\end{equation}
	Setting $f \colon S^m \to \bR, f(x,y):=-\frac{1}{3} \varphi(\|x\|) y^3$, we have $df=\iota^* \alpha_{T^*\bR^m}$.
	We define a locally closed subset $Z$ of $\bR^m \times \bR$ by
	\begin{equation}
	Z:=
	\left\{
	(x,t) \in \bR^m \times \bR \; \middle| \;
	\resizebox{.5\textwidth}{!}{$
	\begin{aligned}
		& \| x \| \le 1, \\
		& -\frac{1}{3} \varphi(\|x\|) (1-\| x \|^2)^{\frac{3}{2}} \le t <\frac{1}{3} \varphi(\|x\|)(1-\| x \|^2)^{\frac{3}{2}}
	\end{aligned}$}
	\right\}
	\end{equation}
	and $F:= \bfk_Z \in \Db(\bR^m \times \bR)$.
	Using the object $F$, one can show $e(\iota(S^m)) \ge e_{\cD(\bR^m)}(F,F) \ge 2/3$ as in \cref{eg:immersionSn}.
	On the other hand, the estimate by Akaho~\cite{Akaho} only gives 
	$e(\iota(S^m)) \ge \min_{r\in [0,\frac{1}{2}]} \{\frac{2}{3} (1-r^2)^\frac{3}{2}\cdot \varphi (r) \}$, which is less than $\sqrt{3}/8$. 
\end{example}

Our theorem is also applicable to non-exact Lagrangian submanifolds. 
We focus on graphs of closed 1-forms here. 

\begin{example}\label{eg:closedoneform}
	Let $M$ be a compact manifold and $\eta_i \colon M \to T^*M$ a closed 1-form for $i=1,2$.
	Set $L_i:=\Gamma_{\eta_i} \subset T^*M$ the graph of $\eta_i$ for $i=1,2$, and assume that $L_1$ and $L_2$ intersect transversally.
	We consider the displacement energy $e(L_1,L_2)$.
	The symplectic diffeomorphism $\psi$ on $T^*M$ defined by $\psi(x;\xi) :=(x;\xi-{\eta_1}(x))$ sends $L_1$ to the zero-section $M$ and $L_2$ to $\Gamma_{\eta_2-\eta_1}$. 
	Thus we assume $L_1=M$ and $L_2=\Gamma_{\eta}$, where $\eta$ is a closed Morse 1-form from the beginning.
	Let $p \colon \tl{M} \to M$ be the abelian covering of $M$ corresponding to the kernel of the pairing with $\eta$.
	Then there exists a function $f \colon \tl{M} \to \bR$ such that $p^*\eta = df$.
	By assumption, $f$ is a Morse function on $\tl{M}$.
	Define a closed subset $Z$ of $\tl{M} \times \bR$ by
	\begin{equation}
	Z
	:=
	\{ (x,t) \in \tl{M} \times \bR \mid f(x)+t \ge 0 \}.
	\end{equation}
	Then we have 
	\begin{equation}
	    F:=R(p \times \id_\bR)_* \bfk_{Z} \in \cD_{L}(M)\quad \text{and}\quad 
	    e(L_1,L_2)\geq e_{\cD(M)}(\bfk_{M \times[0,+\infty)},F)
	\end{equation}
	by \cref{thm:energy}.
	
	Let us consider the estimate for $e_{\cD(M)}(\bfk_{M \times[0,+\infty)},F)$. 
	First, we have
	\begin{align}
	\RHom(\bfk_{M \times [0,+\infty)}, {T_c}_*F)
	& \simeq
	\RHom(\bfk_{\tl{M} \times [-c,+\infty)}, \bfk_{Z}) \\
	\notag & \simeq
	\RG_{\tl M \times [-c,+\infty)}\left( \tl M \times \bR;\bfk_{Z} \right).
	\end{align}
	Define $U_c:= \{ x \in \tl M \mid f(x)> c\} $ for $c\in \bR$. 
	Then the cohomology of the last complex $\RG_{\tl M \times [-c,+\infty)}\left( \tl M \times \bR;\bfk_Z \right)$ is isomorphic to $H^\ast (\tl M, U_c)$ and for $c\le d$, $\tau_{c,d}$ is the canonical morphism induced by the map $(\tl M, U_d)\to (\tl M, U_c)$ of the pairs. 
	Hence this persistence module is isomorphic to $(H^*(\tl M, U_c))_{c\in \bR}$ and it is the dual of the persistence module $(H_*(\tl M, U_c))_{c\in \bR}$. 
	The persistence module $(H_*(\tl M, U_c))_{c\in \bR}$ can be studied by Morse homology theory of $-f$ or  Morse-Novikov theory of $-\eta$. 
	Let $v$ be a vector field on $M$ which is a $(-\eta )$-gradient and satisfies the transversality condition in the sense of  Pajitnov~\cite[Chapter~3 and Chapter~4]{Pajitnov}. 
	The existence and denseness of such vector fields hold (see Pajitnov~\cite[Chapter~4]{Pajitnov}).
	Moreover let $\tl v$ be the lift of $v$ to $\tl M$. 
	The Morse-Novikov complex $C:=C(-\eta, v)$ with respect to $\tl v$ has the filtration $(C_{\le c})_{c\in \bR}$ defined by the values of $-f$.  Here we regard $C$ as a finitely generated free module over the Novikov field 
	\begin{equation}
	\left\{ 
	\sum_{i=1}^\infty c_i T^{\lambda_i} 
	\;\middle| \; 
	\begin{aligned}
	& c_i\in \bfk, \text{$\lambda_i =\int_\gamma \eta$ for some $\gamma \in H_1(M;\bZ)$},\\
	& \lambda_1 < \lambda_2<\cdots, \lim_{i \to \infty} \lambda_i =+\infty 
	\end{aligned}
	\right\}.
	\end{equation} 
	The persistence module $(H_*(C/C_{\le c}))_{c\in \bR}$ is isomorphic to $(H_*(\tl M, U_c))_{c\in \bR}$ by usual Morse theoretic arguments. 
	Each critical point generates or kills rank $1$ subspace of the persistent homology. 
	Hence one can prove that our estimate is greater than or equal to
	\begin{equation}
	\max_p \min_q \left\{ |f(p)-f(q)| 
	\;\middle|\; 
	\begin{aligned}
	& p, q \in \mathrm{Crit}(-f), 
	|\mathrm{ind} (p)-\mathrm{ind} (q)|=1, \\
	& \text{there is a flow of $\tl v$ connecting $p$ and $q$}  
	\end{aligned}
	\right\}, 
	\end{equation}
	where $\mathrm{Crit}(-f)$ is the set of the critical points of $-f$ 
	and $\mathrm{ind}(p)$ is the Morse index of $p \in \mathrm{Crit}(-f)$. 
	
	The persistence module $(H_*(C/C_{\le c}))_{c\in \bR}$ is not finitely generated in the usual sense of persistent homology theory. 
	However we can apply the theory of Usher--Zhang~\cite{UZ16} to $C$.  
	Their result describes the ``barcodes" of the persistence module $(H_*(C_{\le c}))_{c}$ and 
	one can check that our estimate in this case 
	coincides with the length of the longest concise barcodes for  $C(-\eta, v)$ defined in \cite{UZ16}. 
\end{example}

In the last example below, our estimate determines the displacement energy. 

\begin{example}[Special case of \cref{eg:closedoneform}]
	Let $L=\Gamma_\eta \subset T^*S^1$ be the graph of a non-exact 1-form $\eta \colon S^1 \to T^*S^1$.
	Assume that $L$ and the zero-section $S^1$ intersect transversally at only two points. 
	We estimate the displacement energy $e(S^1, L)$. 
	Let $p \colon \bR \to S^1$ be the universal covering and take a function $f$ on $\bR$ such that $df=p^*\eta$.
	Define $F:=R(p \times \id_\bR)_* \bfk_{\{(x,t) \in \bR \times \bR \mid f(x)+t \ge 0\}} \in \cD_{L}(S^1)$.
	Then a similar argument to Example \ref{eg:closedoneform} shows that $e_{\cD(S^1)}(\bfk_{S^1 \times [0,+\infty)},F)$ is equal to the smaller area enclosed by $S^1$ and $L$. 
	One can check that $e(S^1, L)$ is equal to the area. 
\end{example}

\newcommand{\etalchar}[1]{$^{#1}$}
\def\cprime{$'$}


\begin{thebibliography}{CCSG{\etalchar{+}}09}
	
	\bibitem[Aka15]{Akaho}
	M.~Akaho, {\em Symplectic displacement energy for exact {L}agrangian immersions},
	\href{http://arxiv.org/abs/1505.06560}{\texttt{arXiv:1505.06560}},  (2015).
	
	\bibitem[CCSG{\etalchar{+}}09]{CCSGGO}
	F.~Chazal, D.~Cohen-Steiner, M.~Glisse, L.~J. Guibas, and S.~Y. Oudot,
	\newblock {\em Proximity of persistence modules and their diagrams},
	\newblock in: {Proceedings of the Twenty-fifth Annual Symposium on
		Computational Geometry}, SCG '09, pp.~237--246, New York, NY, USA, (2009). %ACM.
	
	\bibitem[CdSGO16]{CdSGO16}
	F.~Chazal, V.~de~Silva, M.~Glisse, and S.~Oudot, {The Structure and
		Stability of Persistence Modules}, SpringerBriefs in Mathematics, Springer,
		(2016).
	
	\bibitem[FOOO09a]{FOOO09}
	K.~Fukaya, Y.-G. Oh, H.~Ohta, and K.~Ono, {Lagrangian Intersection {F}loer
		Theory: Anomaly and Obstruction. {P}art {I}}, Vol.~{46} of 
		{AMS/IP Studies in Advanced Mathematics}, American Mathematical Society,
		Providence, RI; International Press, Somerville, MA, (2009).
	
	\bibitem[FOOO09b]{FOOO092}
	K.~Fukaya, Y.-G. Oh, H.~Ohta, and K.~Ono, {Lagrangian
		intersection {F}loer theory: anomaly and obstruction. {P}art {II}},
	Vol.~{46} of {AMS/IP Studies in Advanced Mathematics}, American
	Mathematical Society, Providence, RI; International Press, Somerville, MA,
		(2009).
	
	\bibitem[FOOO13]{FOOO13}
		K.~Fukaya, Y.-G. Oh, H.~Ohta, and K.~Ono, {\em Displacement of polydisks
		and {L}agrangian {F}loer theory}, {J. Symplectic Geom.} {\bfseries 11}
	(2013), no.~2, 231--268.
	
	\bibitem[GKS12]{GKS}
		S.~Guillermou, M.~Kashiwara, and P.~Schapira, {\em Sheaf quantization of
		{H}amiltonian isotopies and applications to nondisplaceability problems},
		{Duke Math. J.} {\bfseries 161} (2012), no.~2, 201--245.
	
	\bibitem[GS14]{GS14}
		S.~Guillermou and P.~Schapira, {\em Microlocal theory of sheaves and {T}amarkin's
		non displaceability theorem}, in: {Homological Mirror Symmetry and Tropical
		Geometry}, Vol.~{15} of {Lect. Notes Unione Mat. Ital.}, pp.~43--85,
		Springer, Cham, (2014).
	
	\bibitem[Gui12]{Gu12}
		S.~Guillermou, {\em Quantization of conic {L}agrangian submanifolds of cotangent
		bundles}, \href{http://arxiv.org/abs/1212.5818v2}{\texttt{arXiv:1212.5818v2}},  (2012).
	
	\bibitem[Gui16a]{Gulec}
		S.~Guillermou, {\em Quantization of exact
		{L}agrangian submanifolds in a cotangent bundle}, {lecture notes available
		at the author's webpage},  (2016).
	
	\bibitem[Gui16b]{Gu16}
		S.~Guillermou, {\em The three cusps conjecture}, 
		\texttt{arXiv:1603.07876},  (2016).
	
	\bibitem[Hof90]{Hofer90}
		H.~Hofer, {\em On the topological properties of symplectic maps},
		{Proc. Roy. Soc. Edinburgh Sect. A}, {\bfseries 115} (1990), no.~1--2, 25--38.
	
	\bibitem[Ike19]{Ike19}
	    Y.~Ike, Compact exact {L}agrangian intersections in cotangent bundles via sheaf quantization, 
	    {\em Publ. Res. Inst. Math. Sci.}, {\bfseries 55} (2019),
        no.~4, 737--778.
	
	\bibitem[KS90]{KS90}
		M.~Kashiwara and P.~Schapira, {Sheaves on Manifolds}, Vol.~{292}
	of {Grundlehren der Mathematischen Wissenschaften}, Springer-Verlag,
	Berlin, (1990).
	
	\bibitem[KS18]{KS18persistent}
		M.~Kashiwara and P.~Schapira, {\em Persistent homology and microlocal sheaf theory},
	{Journal of Applied and Computational Topology}, {\bfseries 2} (2018),
	no.~1--2, 83--113.
	
	\bibitem[Oh05]{Oh05}
		Y.-G. Oh, {\em Construction of spectral invariants of {H}amiltonian paths on closed
	symplectic manifolds}, in: {The breadth of symplectic and {P}oisson
		geometry}, Vol.~{232} of {Progr. Math.}, pp.~525--570, Birkh\"auser
	Boston, Boston, MA, (2005).
	
	\bibitem[Paj06]{Pajitnov}
	A.~V. Pajitnov, {Circle-Valued {M}orse Theory}, Vol.~{32} of 
	{De Gruyter Studies in Mathematics}, Walter de Gruyter \& Co., Berlin, (2006).
	
	\bibitem[Pol93]{Polterovich93}
		L.~Polterovich, {\em Symplectic displacement energy for {L}agrangian submanifolds},
	{Ergodic Theory Dynam. Systems} {\bfseries 13} (1993), no.~2, 357--367.
	
	\bibitem[PS16]{PS16}
		L.~Polterovich and E.~Shelukhin, {\em Autonomous {H}amiltonian flows, {H}ofer's
	geometry and persistence modules}, {Selecta Math. (N.S.)} {\bfseries 22}
	(2016), no.~1, 227--296.
	
	\bibitem[PSS17]{PSS17}
        L.~Polterovich, E.~Shelukhin, and V.~Stojisavljevi\'{c}, Persistence modules
        with operators in {M}orse and {F}loer theory, {\em Mosc. Math. J.},
        {\bfseries 17} (2017), no.~4, 757--786.
	
	\bibitem[Sch00]{Schwarz00}
		M.~Schwarz, {\em On the action spectrum for closed symplectically aspherical
	manifolds}, {Pacific J. Math.}, {\bfseries 193} (2000), no.~2, 419--461.
	
	\bibitem[Tam18]{Tamarkin}
		D.~Tamarkin, {\em Microlocal condition for non-displaceability},
		in: {Algebraic
	and Analytic Microlocal Analysis}, Vol.~{269} of {Springer
	Proceedings in Mathematics \& Statistics}, pp.~99--223, Springer, Cham, (2018).
	
	\bibitem[UZ16]{UZ16}
		M.~Usher and J.~Zhang, {\em Persistent homology and {F}loer-{N}ovikov theory},
		{Geom. Topol.} {\bfseries 20} (2016), no.~6, 3333--3430.
	
\end{thebibliography}
\end{document}